\documentclass[a4paper, reqno]{amsart}

\usepackage[UKenglish]{babel}
\usepackage{amsmath, amssymb, amsthm}
\usepackage{scrextend}
\usepackage[hidelinks]{hyperref}
\usepackage{xpatch}
\usepackage{color}
\usepackage{tikz-cd}
\usepackage{enumitem}
\usepackage{theoremref}
\usepackage{quiver}
\usepackage{mathtools}
\usepackage{ifthen}

\usepackage[inline,final]{showlabels}
\showlabels{thlabel}

\xpatchcmd{\proof}{\itshape}{\normalfont\bfseries}{}{}
\newtheoremstyle{repeat}{}{}{\itshape}{}{\bfseries}{.}{.5em}{#3, repeated}

\newtheorem{theorem}{Theorem}[section]
\newtheorem{proposition}[theorem]{Proposition}
\newtheorem{lemma}[theorem]{Lemma}
\newtheorem{corollary}[theorem]{Corollary}

\theoremstyle{definition}
\newtheorem{definition}[theorem]{Definition}
\newtheorem{remark}[theorem]{Remark}

\newtheorem{example}[theorem]{Example}
\newtheorem{question}[theorem]{Question}
\newtheorem{fact}[theorem]{Fact}

\theoremstyle{repeat}
\newtheorem*{repeated-theorem}{Repeat}

\newcommand{\E}{\mathcal{E}}
\newcommand{\F}{\mathcal{F}}
\newcommand{\G}{\mathcal{G}}
\newcommand{\B}{\mathcal{B}}

\renewcommand{\L}{\mathcal{L}}
\newcommand{\N}{\mathbb{N}}

\newcommand{\X}{\mathcal{X}}
\newcommand{\U}{\mathcal{U}}
\newcommand{\Z}{\mathbb{Z}}
\newcommand{\R}{\mathbb{R}}

\newcommand{\TMod}[1][T]{#1 \text{--} \mathbf{Mod}} % Usage: \TMod for T-Mod (default) or \TMod[S] for S-Mod
\newcommand{\Set}{\mathbf{Set}}

\newcommand{\Sh}{\mathbf{Sh}}

\newcommand{\Topos}{\mathbf{Topos}}

\newcommand{\Sub}{\operatorname{Sub}}

\newcommand{\colim}{\mathrm{colim}}
\newcommand{\tp}{\operatorname{tp}}

\newcommand{\id}{{id}}
\renewcommand{\phi}{\varphi}

\newcommand{\ec}[1][]{\ifthenelse{\equal{#1}{}}{\textup{ec}}{{#1}\textup{-ec}}}
\renewcommand{\sec}[1][]{\ifthenelse{\equal{#1}{}}{\textup{sec}}{{#1}\textup{-sec}}}
\newcommand{\coh}{{\textup{coh}}}
\newcommand{\comp}{{\textup{comp}}}

\title{Existentially closed models and locally zero-dimensional toposes}

\author{Mark Kamsma and Joshua Wrigley}
\thanks{The first author is supported by the EPSRC grant EP/X018997/1, and both authors are supported by EPSRC grant EP/V028812/1.}
\email[Mark Kamsma]{mark@markkamsma.nl}
\urladdr[Mark Kamsma]{https://markkamsma.nl}
\address[Mark Kamsma]{School of Mathematical Sciences, Queen Mary University of London, London, E1 4NS, UK}
\email[Joshua Wrigley]{j.wrigley@qmul.ac.uk}
\urladdr[Joshua Wrigley]{https://jlwrigley.github.io/}
\address[Joshua Wrigley]{School of Mathematical Sciences, Queen Mary University of London, London, E1 4NS, UK}

\date{\today \\ \indent \emph{2020 Mathematics Subjects Classification}: Primary: 03G30; secondary: 03B20, 03B22, 18B25, 18C10}

\begin{document}

% Abstract
\begin{abstract}
The notion of an existentially closed model is generalised to a property of geometric morphisms between toposes.  We show that important properties of existentially closed models extend to existentially closed geometric morphisms, such as the fact that every model admits a homomorphism to an existentially closed one. Other properties do not generalise: classically, there are two equivalent definitions of an existentially closed model, but this equivalence breaks down for the generalised notion. We study the interaction of these two conditions on the topos-theoretic level, and characterise the classifying topos of the e.c.\ geometric morphisms when the conditions coincide.
\end{abstract}

% Title
\maketitle

% Table of contents
\setcounter{tocdepth}{1} % Only show sections
\tableofcontents

% Contents
\section{Introduction}
Positive model theory replaces the traditional use of full first-order logic with positive logic. This generalises classical model theory, because classical negation can be added back in through \emph{Morleyisation}. Recently, deep model-theoretic ideas have been generalised to the positive context, with negligible concessions in the strength of the results (see e.g., \cite{ben-yaacov_simplicity_2003, ben-yaacov_fondements_2007, haykazyan_spaces_2019, kamsma_bilinear_2023}). Meanwhile, in topos theory ``positive logic'' is known under the name of ``coherent logic'', the finitary fragment of geometric logic. Geometric logic, and thus coherent logic, is known to behave extremely well in toposes, through the study of classifying toposes. With the aim of further cultivating this common ground, we initiate the topos-theoretic study of the protagonists of positive model theory: existentially closed models.

Existentially closed models (hereafter, e.c.\ models) are models of a theory where any (coherent) formula that can be solved in an extension can already be solved in that model itself. The archetypal examples are algebraically closed fields. Any polynomial that can be solved in a field extension of an algebraically closed field already admits a solution in the algebraically closed field itself. In fact, the algebraically closed fields are exactly the e.c.\ models of the theory of fields. In this case, the class of e.c.\ models is first-order axiomatisable, but this is not true in general. It is true, however, that the class of e.c.\ models can always be geometrically axiomatised (see \thref{geometric-axiomatisation-always-possible}). This raises the natural question: what does the classifying topos of the e.c.\ models of a coherent theory look like? A consequence of the axiomatisation is a strong connection between this classifying topos and the notion of zero-dimensionality from topology.

Classically, there are two equivalent definitions of e.c.\ models (cf.\ \thref{existentially-closed-model}). One is in terms of the homomorphisms that have the model as a domain, requiring them to both preserve and reflect the satisfaction of coherent formulas, while the other asserts that whenever a coherent formula $\phi(x)$ does not hold for a tuple of elements $a$, then there is a coherent formula $\psi(x)$, satisfied by $a$, that implies $\neg \psi(x)$ modulo the theory.

We generalise both notions to arbitrary toposes, keeping the name e.c.\ for the former, but naming the latter ``strongly existentially closed'' (or s.e.c.). The distinction is important, because in this generality they are no longer equivalent (\thref{ec-sec-examples}).  The previous appearance of e.c.\ models in the topos-theoretic literature, in \cite{blass_classifying_1983}, uses the latter definition.  However, as we shall see, the former is much better behaved topos-theoretically.

 %We develop basic theory concerning these generalisations, where the notion of e.c.\ turns out to be much better behaved than the notion of s.e.c. The best results are obtained when they coincide for a models of a fixed theory in a fixed topos, such as is the case for our motivating setup with models of a coherent theory in $\Set$. It should also be noted that in \cite{blass_classifying_1983} there is a small part about e.c.\ models. They take the s.e.c.\ version as a definition in arbitrary toposes, but ultimately they only consider such models in $\Set$ for a coherent theory.

\textbf{Main results.}
We leave the precise statement of our main theorem (\thref{ec-subtopos-characterisation}) until after the relevant terminology and notation has been introduced. When only considering coherent theories with models in $\Set$, a simplification is possible, yielding:
\begin{theorem}
\thlabel{ec-subtopos-characterisation-coherent-set}
Let $T$ be a coherent theory. Then following are equivalent for a subtopos $\G \hookrightarrow \Set[T]$ with enough points:
\begin{enumerate}[label=(\roman*)]
\item $\G \simeq \Set[T^{\ec}]$,
\item $\G$ is locally zero-dimensional and entwined with $\Set[T]$,
\item $\G$ is the maximal s.e.c.-subtopos of $\Set[T]$ with enough points.
\end{enumerate}
\end{theorem}
We briefly explain some of the terms in the above statement.
\begin{enumerate}
    \item The theory $T^{\ec}$ is the common geometric theory of all the e.c.\ models of $T$ in $\Set$, i.e.\ $\Set[T^{\ec}]$ is the classifying topos of the e.c.\ models of $T$.
    \item A topos $\G$ is locally zero-dimensional if it has a generating set of objects whose subobject frames are zero-dimensional (\thref{locally-zero-dimensional-topos}).
    \item Let $T'$ be a quotient theory of $T$ classified by $\G$. Then to say that $\G \simeq \Set[T']$ and $\Set[T]$ are `entwined' should be thought of as follows: for every model $M$ of $T$, there is a model $N$ of $T'$ admitting a homomorphism $M \to N$ (note that the converse is automatically satisfied in this case).
    \item Finally, $\G$ is an s.e.c.-subtopos of $\Set[T]$ if the inverse image $i^*(G_T)$ of the generic model $G_T$ under the inclusion $i: \G \hookrightarrow \Set[T]$ is internally s.e.c.
\end{enumerate}  

Along the way, we prove other important facts about our generalised notion of existential closedness.  For model theory, it is important that every model admits a homomorphism into an e.c.\ model, which is a well-known fact for coherent theories in $\Set$. In \thref{every-geometric-morphism-continues-to-ec}, we show that the same is true for e.c.\ models of geometric theories in arbitrary toposes. Finally, based on existing work \cite{blass_classifying_1983,caramello_atomic_2012}, we apply our machinery to topos-theoretically characterise atomic e.c.\ models (\thref{points-factorising-through-double-negation}) and countably categorical theories (\thref{countably-categorical-iff-atomic}) for coherent theories in $\Set$.

\textbf{Overview.}
Sections \ref{sec:preliminaries-and-notation} and \ref{sec:existentially-closed-models} contain preliminary material.  In Section \ref{sec:preliminaries-and-notation}, we establish some notation and recall basic facts about classifying toposes and coherent toposes.  We recall the two classically equivalent definitions of e.c.\ models in Section \ref{sec:existentially-closed-models}, as well as their basic properties.

In Section \ref{sec:existentially-closed-geometric-morphisms}, we generalise from e.c.\ models in $\Set$ to e.c.\ geometric morphisms $\F \to \E$.  We demonstrate that the property of being e.c.\ is well-behaved in that it is \emph{Morita invariant} (\thref{ec-on-generating-is-enough}). We also introduce the classifying topos $\E_{\ec[\F]}$ of e.c.\ geometric morphisms, and the notion of two toposes $\E$ and $\E'$ being $\F$-entwined.  We show that $\F$-entwinement is an equivalence relation on toposes and it is equivalent to having equivalent classifying toposes of e.c.\ morphisms $\E_{\ec[\F]} \simeq \E'_{\ec[\F]}$ (\thref{entwinement-characterisations}). The key ingredient for that last result is the aforementioned fact that, in any topos, any model admits a homomorphism into an e.c.\ model (\thref{every-geometric-morphism-continues-to-ec}).

For models of a coherent theory in $\Set$, the e.c.\ models and s.e.c.\ models coincide.  In Section \ref{sec:sec-geometric-morphisms}, we generalise s.e.c.\ models to arbitrary geometric morphisms $\F \to \E$. However, the notion of s.e.c.\ we introduce is heavily dependent on a particular choice of presentation for $\E$, encapsulated by a choice of a generating set $\X$ and a basis of subobjects $\B$ for $\X$ (see \thref{ec-sec-examples}). When $\E$ is coherent, the notion becomes somewhat tamer in that there are canonical choices for $\X$ and $\B$ (\thref{sec-for-coherent}).

In Section \ref{sec:locally-zero-dimensional-toposes}, we introduce the notion of a locally zero-dimensional topos. Our terminology is justified since a localic topos is locally zero-dimensional if and only if the corresponding locale is locally zero-dimensional (\thref{sheaves-on-locally-zero-dimensional-locale-are-locally-zero-dimensional}).

Section \ref{sec:classifying-topos-ec-models} combines the preceding content in our main result---a characterisation, under suitable hypotheses, of the classifying topos of e.c.\ models. These unspoken hypotheses are satisfied in the motivating case of a coherent theory and models in $\Set$. We also provide a class of examples where the hypotheses are satisfied in a topos other than $\Set$, namely the topos of ``difference sets'' $\Set^\N$ (see \thref{automorphism-ec-iff-underlying-ec,automorphism-examples}).

Finally, in Sections \ref{sec:atomic-ec-models-and-double-negation-sheaves} and \ref{sec:countably-categorical-theories-and-atomic-toposes}, we relate our work to previous literature on similar topics \cite{blass_classifying_1983,caramello_atomic_2012}.  We give a topos-theoretic characterisation of the model-theoretic notion of an atomic e.c.\ model in Section \ref{sec:atomic-ec-models-and-double-negation-sheaves}, and of countably categoricity (in the positive model-theoretic sence) in Section \ref{sec:countably-categorical-theories-and-atomic-toposes}.  In addition, Section \ref{sec:atomic-ec-models-and-double-negation-sheaves} also establishes connections between the property of being a dense subtopos and the various subtoposes that have been studied throughout the paper.

\textbf{Acknowledgements.} We would like to thank our local research group (Ivan Toma\v{s}i\'{c}, Behrang Noohi, Ming Ng, Luka Ilic and Cameron Michie) for many useful discussions.

\section{Preliminaries and notation}
\label{sec:preliminaries-and-notation}
\subsection{Notation and conventions}
We begin by setting out some conventions:
\begin{enumerate}
    \item Throughout, ``topos'' will mean a Grothendieck topos.
    \item We will often omit the arrow when referring to subobjects in a topos, instead writing $A \leq X$ to denote a subobject (i.e.\ an equivalence class of monomorphisms $A \hookrightarrow X$).  Since it will be clear from context, when referring to the subobject lattice of an object, we will also omit the ambient topos, e.g.\ writing $\Sub(X)$ instead of $\Sub_\E(X)$.
    \item When discussing logical formulas, we will follow the standard model-theoretic convention that single letters denote tuples, e.g.\ $x$ and $y$ will denote (finite, potentially empty, potentially of mixed sorts) tuples of variables, and if $M$ is a model in $\Set$, then $a \in M$ refers to a tuple of elements.
    
    The interpretation of a logical formula $\phi(x)$ in a model $M$ will be written as $\phi(M)$.  If $M$ is a model in $\Set$, then $\phi(M)$ is a subset of some product of sorts in $M$ (corresponding to the sorts of $x$).  If $M$ is a model in an arbitrary topos $\F$, $\phi(M)$ is a subobject of a product of sorts.  We will still write $\phi(M) \leq M$ regardless.
\end{enumerate}
\subsection{Classifying toposes}
We now briefly recall the basics of classifying topos theory. Given a geometric theory $T$, a \emph{classifying topos} $\Set[T]$ for $T$ is a topos for which there is a natural equivalence of categories
\begin{equation}\label{classifying-topos-equation}
    \TMod[T](\F) \simeq \Topos(\F,\Set[T]) 
\end{equation}
for any topos $\F$, where $\TMod[T](\F)$ denotes the category of models of $T$ internal to $\F$ and their homomorphisms, while $\Topos(\F,\Set[T])$ denotes the category of geometric morphisms $\F \to \Set[T]$ and natural transformations between these. Every geometric theory has a classifying topos, and every topos is the classifying topos for some geometric theory.  The universal property satisfied by the classifying topos of $T$ ensures that it is defined uniquely up to equivalence.

More explicitly, the classifying topos $\Set[T]$ contains an internal model of $T$ called the \emph{generic model} $G_T$. The above correspondence is then given by sending a geometric morphism $p: \F \to \Set[T]$ to the model $p^*(G_T)$ of $T$ in $\F$.
\begin{fact} We note some important facts about the generic model:
    \begin{enumerate}[label = (\roman*)]
        \item The interpretations of formulas in the generic model, i.e.\ the objects $\phi(G_T)$, generate the topos $\Set[T]$;
        \item The subobjects in $\Set[T]$ of $\phi(G_T)$ are precisely the subobjects of the form $\psi(G_T) \leq \phi(G_T)$ (\cite[Lemma D1.4.4(iv)]{johnstone_sketches_2002}), where $\psi$ is also a geometric formula, and moreover there is an inclusion of subobjects $\psi(G_T) \leq \phi(G_T)$ if and only if $T$ proves the sequent $\psi(x) \vdash_x \phi(x)$.
    \end{enumerate}
\end{fact}

\subsection{Coherent toposes}
The assertion that every topos classifies some theory necessitates the consideration of geometric logic.  To focus on solely coherent theories, we can restrict to \emph{coherent toposes}.  These are toposes that are, in a strong sense, determined by their compact objects.  In this paper, we will often restrict to the case of a coherent topos to simplify our conditions.
\begin{definition}[{\cite[Section D3.3]{johnstone_sketches_2002}}]
~
\begin{enumerate}[label=(\roman*)]
\item An object $X$ in a topos is said to be \emph{compact} if every jointly epimorphic family $\{Y_i \to X\}_{i \in I}$ admits a jointly epimorphic finite subfamily.
\item An object $X$ is \emph{coherent} if, in addition to being compact, it is \emph{stable} in the sense that for any pair of arrows $Y \to X \leftarrow Z$ with compact domains, their pullback $Y \times_X Z$ is also compact.
\item A topos is said to be \emph{coherent} if it has a generating set of coherent objects.
\end{enumerate}
\end{definition}
\begin{fact}
\thlabel{coherent-topos-facts}
We recall some facts about coherent toposes.
\begin{enumerate}[label=(\roman*)]
\item A topos is coherent if and only if it is the classifying topos of a coherent theory (\cite[Theorem D3.3.1]{johnstone_sketches_2002}).  Indeed, if $T$ is coherent, then the objects $\phi(G_T)$, where $\phi(x)$ is a coherent formula, are coherent (and generate the topos).
\item In a coherent topos there is (up to isomorphism) only a set of coherent objects.
\item A coherent topos $\E$ also \emph{has enough points} (see e.g., \cite[\S IX.11]{maclane_sheaves_1994}), meaning that the class of geometric morphisms $\{p: \Set \to \E\}$ is jointly surjective (i.e.\ the inverse image functors $\{p^*: \E \to \Set\}$ are jointly faithful/conservative).  If $\E$ classifies a theory $T$, this is equivalent to stating that $T$ satisfies a completeness theorem---if $\phi(M) = \psi(M)$ for every model $M$ in $\Set$, then $\phi$ and $\psi$ are $T$-provably equivalent.
\end{enumerate}
\end{fact}
Generalising the above scenario, we will say that a topos $\E$ \emph{has enough $\F$-points} to mean that the class of geometric morphisms $\{\F \to \E\}$ is jointly surjective.
\section{Existentially closed models}
\label{sec:existentially-closed-models}
The motivating objects for this paper are the existentially closed models (in $\Set$) of a coherent theory. In this section, we recall the basics of these objects, to inspire and motivate definitions and results later on. As such, we only consider structures and models in $\Set$ for this section.
\begin{definition}
\thlabel{homomorphism-immersion}
Recall that a function $f: M \to N$ between two structures in the same language is called a \emph{homomorphism} if it preserves satisfaction of coherent formulas. That is, for all $a \in M$ and all coherent $\phi(x)$ we have that
\[
M \models \phi(a) \implies N \models \phi(f(a)).
\]
If $f$ also reflects satisfaction of coherent formulas then we call it an \emph{immersion}. That is, $f$ is an immersion if for all $a \in M$ and all coherent $\phi(x)$ we have that
\[
M \models \phi(a) \quad \Longleftrightarrow \quad N \models \phi(f(a)).
\]
\end{definition}
By induction on the structure of formulas, one easily sees that $f$ is a homomorphism if and only if it preserves satisfaction of all geometric formulas, if and only if it preserves all relations, constants and function symbols. Similarly, a homomorphism is an immersion if and only if it reflects satisfaction of all geometric formulas.
\begin{remark}
\thlabel{immersion-pullback}
Equivalently, a function $f: M \to N$ defines a homomorphism if and only if, for any geometric formula $\phi$, there is a factorisation
\[
\begin{tikzcd}
    \phi(M) \ar[dashed]{r} \ar[hook]{d} & \phi(N) \ar[hook]{d} \\
    M \ar{r}{f} & N,
\end{tikzcd}
\]
and $f$ is an immersion if and only if this square is also a pullback of sets (where the map $f: M \to N$ is interpreted in the appropriate sense, i.e.\ a product of the arrows on the respective sorts).
\end{remark}
\begin{remark}
 An immersion $f: M \to N$ is always injective since
\[
N \models f(a) = f(b) \implies M \models a =b.
\]
In fact, by the same reasoning, every immersion is an embedding (i.e.\ a map between structures that preserves and reflects satisfaction of all \emph{atomic} formulas).
\end{remark}
\begin{definition}
\thlabel{existentially-closed-model}
Let $T$ be a coherent theory and let $M \models T$.
\begin{enumerate}[label=(\roman*)]
\item We call $M$ \emph{existentially closed} (or \emph{e.c.}\ for short) if every homomorphism $f: M \to N$ with $N \models T$ is an immersion.
\item We call $M$ \emph{strongly existentially closed} (or \emph{s.e.c.}\ for short) if, for every coherent formula $\phi(x)$ and every $a \in M$ such that $M \not \models \phi(a)$, there is a coherent formula $\psi(x)$ such that $T$ proves the formula $\neg \exists x(\phi(x) \wedge \psi(x))$ (or equivalently, the sequent $\phi(x) \land \psi(x) \vdash_x \bot$) and $M \models \psi(a)$.
\end{enumerate}
Note that the definition of e.c.\ and s.e.c.\ depends on the theory $T$.  This will often be omitted when the context is obvious.
\end{definition}
It turns out that these definitions, for models of coherent theories in $\Set$, are equivalent. Indeed, both characterisations are often used for defining e.c.\ models in positive logic (also known as coherent logic). However, they both admit generalisations to models of geometric theories in arbitrary toposes, where the equivalence breaks down (see \thref{ec-sec-examples}).
\begin{fact}[{\cite[Lemme 14]{ben-yaacov_fondements_2007}}]
\thlabel{classical-ec-is-sec}
A model of a coherent theory $T$ in $\Set$ is e.c.\ if and only if it is s.e.c.
\end{fact}
\begin{remark}
\thlabel{ec-vs-pec}
Some authors reserve the name ``existentially closed'' for the situation where every relation symbol (including equality) has a negation modulo the theory, or equivalently, when working with respect to embeddings rather than homomorphisms. The term ``positive (existentially) closed'' is then used for what we call ``existentially closed''.
\end{remark}
\begin{example}
\thlabel{classical-ec-examples}
We give some examples of classes of e.c.\ models.
\begin{enumerate}[label=(\roman*)]
\item Let our language consist of infinitely many constant symbols, say $\{c_n\}_{n \in \N}$, and let $T$ declare all these symbols to be distinct. Then the only e.c.\ model is the model whose underlying set consists only of the constants (which are all distinct), e.g.\ $\N$ with $c_n$ interpreted as $n$ for all $n \in \N$. This is because if a model $M$ contains an element $a \in M$ that is not the interpretation of a constant then we can define a homomorphism $f: M \to M$ that sends $a$ to (the interpretation of) $c_0$ and fixes everything else. Then $f$ is not an immersion, as $M \models f(a) = f(c_0)$ but $M \not \models a = c_0$.
\item In the theory of fields, the e.c.\ models are the algebraically closed fields \cite[Section 3.3.4]{tent_course_2012}.
\item Let $T$ be the theory of graphs, with a symbol for inequality $\neq$ and a binary symbols for both the edge relation and its complement (so that homomorphisms between models of $T$ are graph embeddings, i.e.\ injective and they preserve and reflect the edge relation). Then the existentially closed models of $T$ are exactly those (infinite) graphs $G$ such that, for any two finite disjoint sets of vertices $A, B \subseteq G$, there is a vertex $v \in G \setminus (A \cup B)$ with an edge to every vertex in $A$ and no edge to any vertex in $B$. In particular, the unique countable such graph (up to isomorphism) is the \emph{random graph} or \emph{Rado graph} \cite[Exercise 3.3.1]{tent_course_2012}.
\end{enumerate}
\end{example}
\begin{remark}
\thlabel{geometric-axiomatisation-always-possible}
It is well-known that the class of e.c.\ models can not always be axiomatised in finitary logic. Not even in full first-order logic (or equivalently: in coherent logic and allowing an expansion of the language, such as Morleyisation). See for example \thref{classical-ec-examples}(i), where one quickly runs into issues with compactness. However, using the s.e.c.\ characterisation together with \thref{classical-ec-is-sec}, we see that the class of e.c.\ models of a coherent theory $T$ is always geometrically axiomatisable, by adding to $T$ a sequent
\[
\top \vdash_x \phi(x) \vee \bigvee \{ \psi(x) \text{ a coherent formula} \mid T \text{ proves } \phi(x) \wedge \psi(x) \vdash_x \bot \},
\]
for every coherent formula $\phi(x)$.
\end{remark}
Finally, there is the following important fact. It guarantees not only that e.c.\ models exist, but also that when we care about eventual behaviour of models (as model theorists often do), then we only need to look at the e.c.\ models of a theory.
\begin{fact}[{\cite[Theoreme 1]{ben-yaacov_fondements_2007}}]
\thlabel{every-model-continues-to-ec-model}
Every model of a coherent theory $T$ admits a homomorphism into an e.c.\ model of $T$.
\end{fact}

\section{Existentially closed geometric morphisms}
\label{sec:existentially-closed-geometric-morphisms}
We now generalise the notion of being existentially closed from models to geometric morphisms. Any topos $\E$ can be seen as the classifying topos of some geometric theory $T$, so geometric morphisms $\F \to \E$ correspond to models of $T$ in $\F$, while natural transformations between such geometric morphisms correspond to homomorphisms between the corresponding models. 

As interpretations of geometric formulas correspond to subobjects (of the generic model) in $\E$, and following \thref{immersion-pullback}, the following seems a straightforward generalisation of the definitions in Section \ref{sec:existentially-closed-models}.
\begin{definition}
\thlabel{ec-geometric-morphisms}
Let $p, q: \F \to \E$ be geometric morphisms. A natural transformation
\[\begin{tikzcd}
\F && \E,
    \arrow[""{name=0, anchor=center, inner sep=0}, "p", curve={height=-18pt}, from=1-1, to=1-3]
    \arrow[""{name=1, anchor=center, inner sep=0}, "q"', curve={height=18pt}, from=1-1, to=1-3]
    \arrow["\eta", shorten <=5pt, shorten >=5pt, Rightarrow, from=0, to=1]
\end{tikzcd}
\]
is called an \emph{immersion} if for any monomorphism $a: A \hookrightarrow X$ in $\E$, the induced commuting square
\[
\begin{tikzcd}
    p^*(A) \ar[hook]{d}[']{p^*(a)} \ar{r}{\eta_A} & q^*(A) \ar[hook]{d}{q^*(a)} \\
    p^*(X) \ar{r}{\eta_X} & q^*(X)
\end{tikzcd}
\]
is a pullback.

We say that a geometric morphism $p: \F \to \E$ is \emph{existentially closed} (or \emph{e.c.}\ for short) if every natural transformation to any geometric morphism $q: \F \to \E$ is an immersion.
\end{definition}
We gave an explanation for the reasoning behind this condition, at least when the object $X$ is the interpretation of a geometric formula in the generic model of the theory that is classified by $\E$. It turns out that when this condition is met for a generating set of objects in $\E$ (such as the set of interpretations of formulas in the generic model), then it will hold for all objects in $\E$. This makes being e.c.\ presentation independent, and so Morita-equivalent theories have the same categories of e.c.\ models. In particular, we see that our terminology is justified, because in the case of $\F = \Set$ and coherent $T$ it specialises back to the existing terminology (\thref{ec-as-model-is-ec-as-point}).
\begin{proposition}
\thlabel{ec-on-generating-is-enough}
Let $\eta: p \Rightarrow q$ be a natural transformation of geometric morphisms $p, q: \F \to \E$, then $\eta$ is an immersion if the condition in \thref{ec-geometric-morphisms} is met for all monomorphisms whose codomains are contained in some generating set of objects for $\E$. In particular, it suffices to check existential closedness for a geometric morphism on a generating set.
\end{proposition}
\begin{proof}
This follows from the fact that, in a Grothendieck topos $\E$, colimits are stable under pullback (see \cite{giraud_exactness} or \cite[Appendix]{maclane_sheaves_1994}), i.e.\ 
\[
    X \times_Y \colim_{i \in I} Z_i \cong \colim_{i \in I} (X \times_Y Z_i),
\]
and that every object in $\E$ is a colimit of objects in the generating set $\X$ of $\E$.

Let $X \cong \colim_{i \in I} X_i$, where $X_i \in \X$. Then for any monomorphism $A \leq X$ in $\E$, there is a chain of isomorphisms
\[
    A \cong A \times_X X \cong A \times_X \colim_{i \in I} X_i \cong \colim_{i \in I} A_i,
\]
where $A_i$ denotes the pullback $A \times_{X} X_i$.

By composing the pullback squares
\[
\begin{tikzcd}
    p^*(A_i) \ar{r} \ar[hook]{d} & q^*(A_i) \ar{r} \ar[hook]{d} & q^*(A) \ar[hook]{d} \\
    p^*(X_i) \ar{r} & q^*(X_i) \ar{r} & q^*(X)
\end{tikzcd}
\]
(the left-hand square is a pullback by hypothesis, the right-hand square is a pullback since $q^*$ preserves pullbacks), we have that $p^*(A_i) \cong p^*(X_i) \times_{q^*(X)} q^*(A)$, and hence we conclude that
\begin{align*}
    p^*(X) \times_{q^*(X)} q^*(A) &\cong (\colim_{i \in I} p^*(X_i)) \times_{q^*(X)} q^*(A), \\
    &\cong \colim_{i \in I}(p^*(X_i) \times_{q^*(X)} q^*(A)), \\
    &\cong \colim_{i \in I} p^*(A_i) \cong p^*(A)
\end{align*}
as desired (using the fact that $p^*$ commutes with colimits).
\end{proof}
\begin{remark}
The proof given in \thref{ec-on-generating-is-enough} is valid if we replace `monomorphism' in \thref{ec-geometric-morphisms} with any other property of an arrow that is stable under pullback in a topos, e.g.\ epimorphism.
\end{remark}
\begin{corollary}
\thlabel{ec-as-model-is-ec-as-point}
A model $M$ of a coherent theory $T$ in $\Set$ is e.c.\ according to \thref{existentially-closed-model} if and only if the corresponding geometric morphism $p: \Set \to \Set[T]$ is e.c.\ according to \thref{ec-geometric-morphisms}.
\end{corollary}
\begin{proof}
By definition $M$ is e.c.\ if and only if every homomorphism $f: M \to N$ into a model of $T$ is an immersion. Letting $q: \Set \to \Set[T]$ be the geometric morphism corresponding to $N$, we can view $f$ as a natural transformation $p \Rightarrow q$. Then by \thref{immersion-pullback} we see that $f$ is an immersion in the classical sense (\thref{homomorphism-immersion}) if it is an immersion as a natural transformation in the sense of \thref{ec-geometric-morphisms}. The converse follows from \thref{ec-on-generating-is-enough}, because the objects of the form $\phi(G_T)$ generate $\Set[T]$ and by \thref{immersion-pullback} (and the pullback lemma) the naturality squares involving their subobjects, which are of the form $\phi \land \psi(G_T) \leq \phi(G_T)$, are pullbacks.
\end{proof}
We have already seen examples of e.c.\ models, and thus e.c.\ geometric morphisms from $\Set$, in \thref{classical-ec-examples}. We defer further examples to \thref{ec-sec-examples}, once we have also generalised the notion of being strongly e.c.\ to geometric morphisms, so that we can compare the two.
\begin{proposition}
\thlabel{comp-with-surj-is-ec-implies-ec}
Given a geometric morphism $p: \F \to \E$ and a jointly conservative family of geometric morphisms $\{q_i: \G \to \F\}_{i \in I}$, if the composite $pq_i$ is e.c.\ for each $i \in I$, then so too is $p$.
\end{proposition}
\begin{proof}
Let $\eta: p \Rightarrow r$ be a natural transformation into any geometric morphism $r: \F \to \E$. Then we obtain a natural transformation $q_i^*(\eta): p q_i \Rightarrow r q_i$, for each $i \in I$. As $pq_i$ is e.c., for any subobject $A \leq X$ in $\E$, we have that $q_i^* p^*(A) = (q^*(\eta_X))^{-1}(q_i^* r^*(A)) = q_i^*(\eta_X^{-1}(r^*(A)))$.  As $\{q_i: \G \to \F\}_{i \in I}$ is a jointly conservative family, we conclude that $p^*(A) = \eta_X^{-1}(r^*(A))$, and therefore that $p$ is e.c.
\end{proof}
\subsection{The subtopos of existentially closed points}
Given a geometric morphism $p: \F \to \E$ we write its (surjection, embedding)-factorisation as $\F \twoheadrightarrow p(\F) \hookrightarrow \E$. In particular, $p(\F)$ is a subtopos of $\E$, called the \emph{image of $\F$ under $p$}.
\begin{definition}
\thlabel{ec-subtopos}
Given toposes $\F$ and $\E$, let $\E_{\F\text{-}\ec} \hookrightarrow \E$ denote the subtopos
\[
\bigcup \{ p(\F) \mid p: \F \to \E \text{ an e.c.\ geometric morphism} \},
\]
where the union is taken in the (complete) lattice of subtoposes of $\E$.
\end{definition}
If we view $\E$ as the classifying topos $\Set[T]$ on some theory $T$ then $\E_{\F\text{-}\ec}$ is the classifying topos of the common theory of the e.c.\ models of $T$ in $\F$, that is $\E_{\F\text{-}\ec} \simeq \Set[T^{\F\text{-}\ec}]$ and
\[
T^{\ec[\F]} = \{\sigma \text{ a geometric sequent} \mid \sigma \text{ is valid in every e.c.\ model of $T$ in $\F$}\}.
\]
\begin{proposition}
The inclusion $\E_{\F\text{-}\ec} \hookrightarrow \E$ is itself e.c.
\end{proposition}
\begin{proof}
The collection of geometric morphisms $p: \F \to \E_{\ec[\F]}$, such that the composition $\F \xrightarrow{p} \E_{\ec[\F]} \hookrightarrow \E$ is e.c., is by construction jointly conservative. The statement thus follows from \thref{comp-with-surj-is-ec-implies-ec}.
\end{proof}
\subsection{Entwined toposes}
Different geometric theories can have the same category of e.c.\ models. For example, following \thref{classical-ec-examples}(ii), the theories of fields and algebraically closed fields have the same e.c.\ models in $\Set$, namely the algebraically closed fields. Since we are interested in theories up to their e.c.\ models, we introduce the notion of entwinement.  As we will see, two theories have equivalent categories of e.c.\ models if and only if they are entwined.
\begin{definition}
\thlabel{entwined}
Fix a topos $\F$. We say that toposes $\E$ and $\E'$ are \emph{$\F$-entwined} if there is a pair of geometric embeddings $i: \E \hookrightarrow \G$ and $i': \E' \hookrightarrow \G$ such that, for every geometric morphism $p: \F \to \E$, there is a geometric morphism $q: \F' \to \E'$ and a natural transformation $\eta: ip \Rightarrow i'q$, as in the diagram
\[
\begin{tikzcd}[row sep=tiny]
	& \E \\
	\F && \G, \\
	& {\E'}
	\arrow["i", curve={height=-6pt}, hook, from=1-2, to=2-3]
	\arrow["{ \eta}", Rightarrow, from=1-2, to=3-2]
	\arrow["p", curve={height=-6pt}, from=2-1, to=1-2]
	\arrow["q"', dashed, curve={height=6pt}, from=2-1, to=3-2]
	\arrow["i'"', curve={height=6pt}, hook, from=3-2, to=2-3]
\end{tikzcd}
\]
and vice versa.
\end{definition}
\begin{remark}
    The definition of entwinement is motivated by the model-theoretic notion of a \emph{model companion} $T'$ of a theory $T$ (see e.g.\ \cite[Definition 3.2.8]{tent_course_2012}). Recall that this is a theory $T'$ such that is \emph{model complete} (all its models are existentially closed with respect to embeddings) and such that every model of $T$ embeds into one of $T'$ and vice versa.  Entwinement is the weaker condition where models merely admit a homomorphism to one another.
\end{remark}
\begin{lemma}
\thlabel{entwinement-is-equivalence-relation}
Being $\F$-entwined, for fixed $\F$, is an equivalence relation on toposes.
\end{lemma}
\begin{proof}
Only transitivity requires an argument. Suppose that $\E$ and $\E'$ are $\F$-entwined, as witnessed by $\E \hookrightarrow \G \hookleftarrow \E'$, and $\E'$ and $\E''$ are $\F$-entwined, as witnessed by $\E' \hookrightarrow \G' \hookleftarrow \E''$. By \cite[Scholium B3.4.4]{johnstone_sketches_2002} we can consider the pushout of $\G \hookleftarrow \E' \hookrightarrow \G'$ to obtain a diagram of inclusions of toposes:
\[
\begin{tikzcd}
\E \arrow[r, hook]                    & \G \arrow[rd, hook]  &      \\
\E' \arrow[rd, hook] \arrow[ru, hook] &                      & \G^* . \\
\E'' \arrow[r, hook]                  & \G' \arrow[ru, hook] &     
\end{tikzcd}
\]
It is then straightforward to see that the cospan $\E \hookrightarrow \G^* \hookleftarrow \E''$ witnesses that $\E$ and $\E''$ are $\F$-entwined.
\end{proof}
The following makes precise how entwinement sees toposes (which can be viewed as geometric theories) up to e.c.\ models (i.e., e.c.\ geometric morphisms into them).

\begin{proposition}
\thlabel{entwinement-characterisations}
Let $\F$, $\E$ and $\E'$ be three toposes. Then $\E$ and $\E'$ are $\F$-entwined if and only if $\E_{\ec[\F]} \simeq \E'_{\ec[\F]}$.
\end{proposition}
It follows from our proof that the category of e.c.\ geometric morphisms $\F \to \E$ is equivalent to the category of e.c.\ geometric morphisms $\F \to \E'$. The converse holds if we additionally specify how this equivalence is induced (i.e., by composing with arrows into some fixed $\G$ witnessing entwinement), but we will not make use of this.

The main ingredient in proving \thref{entwinement-characterisations} is the following generalisation of \thref{every-model-continues-to-ec-model}, which is of interest on its own as it shows that e.c.\ geometric morphisms are abundant.
\begin{theorem}
\thlabel{every-geometric-morphism-continues-to-ec}
For any geometric morphism $p: \F \to \E$, there is an e.c.\ geometric morphism $q: \F \to \E$ together with a natural transformation $\eta: p \Rightarrow q$.
\end{theorem}
We note that neither $q$ nor $\eta$ is in any way unique. For example, a field embeds into many different algebraically closed fields of various trascendence degrees, and can embed in such fields in many different ways (cf.\ \thref{classical-ec-examples}(ii)).
\begin{proof}
Essentially the standard construction for models in $\Set$ of a coherent theory, as found in \cite[Theoreme 1]{ben-yaacov_fondements_2007}, can be adapted using the fact that $\F$ is locally $\lambda$-presentable for some $\lambda$ \cite[Corollary D2.3.7]{johnstone_sketches_2002}. For that reason, we temporarily adopt a more logical notation---let $T$ be a geometric theory such that $\E$ classifies $T$, i.e.\ $\E \simeq \Set[T]$. Then $p$ corresponds to a model $M$ of $T$ internal to $\F$. We will build an e.c.\ model $N$ in $\F$, together with a homomorphism $M \to N$ by constructing $N$ as the colimit of a chain of $T$-models $M = M^0 \to M^1 \to \ldots$ of length $\lambda$, which we construct by transfinite induction. This is enough, as we can then take $q$ to be the geometric morphism corresponding to $N$ and $\eta$ the natural transformation corresponding to the homomorphism $M \to N$.

We take $M_0$ as $M$, and at limit stages, we take the colimit of the chain we have constructed so far, which just leaves the successor step.  Having constructed $M^i$ we consider the set of all pairs $(\phi(x), a)$, where $\phi(x)$ is a geometric formula in the language of $T$ and $a: A \to M^i$ is an arrow with $\lambda$-presentable domain. There is indeed only a set of those, because (up to provable equivalence) there is only a set of geometric formulas and (up to isomorphism) there is also only a set of $\lambda$-presentable objects. Enumerate the members of this set of pairs as $\{(\phi_j(x), a_j)\}_{j < \kappa_i}$. We will construct $M^{i+1}$ as the colimit of a chain of $T$-models $M^i = M^i_0 \to M^i_1 \to \ldots$ of length $\kappa_i$. Again, at limit stages, we take the colimit of the chain we have constructed so far, which leaves us with the successor step.

Having constructed $M^i_j$, we consider two cases. If $A \xrightarrow{a_j} M^i \to M^i_j$ does not factor through $\phi_j(M_j^i)$, but there is a homomorphism $h: M^i_j \to M'$ into a model of $T$ such that $A \xrightarrow{a_j} M^i \to M^i_j \xrightarrow{h} M'$ factors through $\phi_j(M')$ then we take $M^i_{j+1} = M'$ with the homomorphism $M^i_j \to M^i_{j+1}$ to be $h$. If this is not the case then we simply take $M^i_{j+1} = M^i_j$ and the relevant arrow is the identity.

This completes the construction of $N$. We now show that $N$ is indeed e.c. Let $\phi(x)$ be a geometric formula in the language of $T$ and let $g: N \to N'$ be a homomorphism into a model of $T$. Suppose we are given arrows $a: A \to N$ and $A \to \phi(N')$ making the diagram below commute, we desire an arrow $\begin{tikzcd}[column sep = small]
   &[-18pt] A \ar[dashed]{r} & \phi(N) &[-18pt]
\end{tikzcd}$ making the whole diagram
\begin{equation}\label{eq:maybe_a_pb:continue-to-ec-model}
    \begin{tikzcd}
A \arrow[rrd, bend left=2em] \arrow[rdd, "a"', bend right] \arrow[rd, dashed] &   & \\
 & \phi(N) \arrow[r] \arrow[d, hook] & \phi(N') \arrow[d, hook] \\
& N \arrow[r, "g"] & N'   
\end{tikzcd}
\end{equation}
commute.  Note that, since $\phi(N) \hookrightarrow N$ is a monomorphism, such an arrow is automatically unique, and so this is enough to conclude that the bottom square is a pullback, and hence that $N$ is e.c. as desired.

We first perform this for the case where $A$ is $\lambda$-presentable. In that case the arrow $a: A \to N$ will factor as $A \xrightarrow{a'} M^i \to N$ for some $i < \lambda$. Let $j < \kappa_i$ be such that $(\phi(x), a') = (\phi_j(x), a_j)$. We distinguish two cases to show that $a$ factors through $\phi(N) \hookrightarrow N$.
\begin{enumerate}[label=\arabic*.]
\item If $A \xrightarrow{a_j} M^i \to M^i_j$ factors through $\phi({M^i_j}) \hookrightarrow M^i_j$ then it follows immediately by composing with the coprojections $M^i_j \to M^{i+1} \to N$.
\item If $A \xrightarrow{a_j} M^i \to M^i_j$ does not factor through $\phi({M^i_j} ) \hookrightarrow M^i_j$ then 
\[M^i_j \to M^{i+1} \to N \xrightarrow{g} N'\]
yields a homomorphism, let's denote it by $k$, into a model of $T$ such that the composite $A \xrightarrow{a_j} M^i \to M^i_j \xrightarrow{k} N'$ factors through $\phi(N') \hookrightarrow N'$. So by construction we must have that $A \xrightarrow{a_j} M^i \to M^i_j \to M^i_{j+1}$ factors through $\phi({M^i_{j+1}}) \hookrightarrow M^i_{j+1}$. The claim then follows by composing with the coprojections $M^i_{j+1} \to M^{i+1} \to N$.
\end{enumerate}
We thus have an arrow $\begin{tikzcd}[column sep = small]
   &[-15pt] A \ar[dashed]{r} & \phi(N) &[-15pt]
\end{tikzcd}$ such that in the diagram (\ref{eq:maybe_a_pb:continue-to-ec-model}) the bottom triangle, the square and the outer square commute. A simple diagram chase together with the fact that $\phi(N') \hookrightarrow N'$ is a monomorphism then shows that the top triangle also commutes, as desired.

Now let $A$ be any object of $\F$, which can therefore by written as the $\lambda$-directed colimit of $\lambda$-presentable objects $A = \colim_{i \in  I} A_i$. By composing with the coprojection $A_i \to A$, each $A_i$ can be fitted into a diagram as in (\ref{eq:maybe_a_pb:continue-to-ec-model}) and thus admits a commuting arrow $\begin{tikzcd}[column sep = small]
   &[-15pt] A_i \ar[dashed]{r} & \phi(N) &[-15pt]
\end{tikzcd}$. The uniqueness of these arrows makes $\phi(N)$ into a cocone for $(A_i)_{i \in I}$. Thus, by the universal property of the colimit, we obtain the required arrow $\begin{tikzcd}[column sep = small]
   &[-15pt] A \ar[dashed]{r} & \phi(N) &[-15pt]
\end{tikzcd}$.
\end{proof}
\begin{corollary}\thlabel{E-and-E_Fec-are-entwined}
    The toposes $\E$ and $\E_{\ec[\F]}$ are $\F$-entwined.
\end{corollary}
\begin{proof}
    Both toposes $\E$ and $\E_{\ec[\F]}$ are subtoposes of $\E$.  Vacuously, every geometric morphism $\F \to \E_{\ec[\F]} \hookrightarrow \E$ admits a natural transformation to a geometric morphism $\F \to \E =\joinrel= \E$.  The converse is provided by \thref{every-geometric-morphism-continues-to-ec}.
\end{proof}

\begin{lemma}
\thlabel{horizontal-factor-for-subtopos}
If $i: \E \hookrightarrow \G$ is a subtopos, then any transformation between composite geometric morphisms $\F \xrightarrow{p} \E \xhookrightarrow{i} \G$ and $\F \xrightarrow{q} \E \xhookrightarrow{i} \G$ admits a horizontal factorisation, i.e.\
\[
\begin{tikzcd}[row sep=tiny]
	& \E \\
	\F && \G \\
	& {\E}
	\arrow["i", curve={height=-6pt}, hook, from=1-2, to=2-3]
	\arrow["{ \eta}", Rightarrow, from=1-2, to=3-2]
	\arrow["p", curve={height=-6pt}, from=2-1, to=1-2]
	\arrow["q"', curve={height=6pt}, from=2-1, to=3-2]
	\arrow["i"', curve={height=6pt}, hook, from=3-2, to=2-3]
\end{tikzcd}
    =
\begin{tikzcd}
	\F && \E & \G.
	\arrow[""{name=0, anchor=center, inner sep=0}, "p", curve={height=-25pt}, from=1-1, to=1-3]
	\arrow[""{name=1, anchor=center, inner sep=0}, "q"', curve={height=25pt}, from=1-1, to=1-3]
	\arrow["i", hook, from=1-3, to=1-4]
	\arrow["{\eta'}", shorten <=5pt, shorten >=5pt, Rightarrow, from=0, to=1]
\end{tikzcd}
\]
\end{lemma}
\begin{proof}
    Recall that $i: \E \hookrightarrow \G$ being a subtopos means that $\E$ is a full reflective subcategory of $\G$.  Hence, the component $\eta'_X$ at an object $X \in \E$ can be taken as $p^*(X) \cong p^* i^* i_* (X) \xrightarrow{\eta_{i_*(X)}}  q^* i^* i_* (X) \cong q^*(X)$. 
\end{proof}

\begin{proof}[Proof of \thref{entwinement-characterisations}]
We first prove the left to right direction. Let $i: \E \hookrightarrow \G$ and $j: \E' \hookrightarrow \G$ witness that $\E$ and $\E'$ are $\F$-entwined. We first show that if $p: \F \to \E$ is e.c., then $\F \xrightarrow{p} \E \xhookrightarrow{i} \G$ factors as $\F \xrightarrow{p'} \E' \xhookrightarrow{j} \G$, after which we will show that the factor $\F \xrightarrow{p'} \E'$ is also e.c. 

Recall, from \cite[Proposition A4.3.11]{johnstone_sketches_2002} say, that the subtopos $j: \E' \hookrightarrow \G$ is induced by a \emph{closure operator} $c$ on subobjects in $\G$, and that $\F \xrightarrow{p} \E \xhookrightarrow{i} \G$ factors through $j: \E' \hookrightarrow \G$ if and only if $p^*i^*(m)$ is an isomorphism for every $c$-dense monomorphism.

By $\F$-entwinement, there is a diagram
\[
\begin{tikzcd}
	& \E \\
	\F & {\E'} & \G. \\
	& \E
	\arrow["\eta", Rightarrow, from=1-2, to=2-2]
	\arrow["i", curve={height=-6pt}, hook, from=1-2, to=2-3]
	\arrow["p", curve={height=-6pt}, from=2-1, to=1-2]
	\arrow["q", from=2-1, to=2-2]
	\arrow["r"', curve={height=6pt}, from=2-1, to=3-2]
	\arrow["{j}", hook, from=2-2, to=2-3]
	\arrow["\theta", Rightarrow, from=2-2, to=3-2]
	\arrow["i"', curve={height=6pt}, hook, from=3-2, to=2-3]
\end{tikzcd}
\]
Let $m: A \hookrightarrow X$ be a $c$-dense monomorphism in $\G$. In the commuting diagram
\[
\begin{tikzcd}
p^* i^* (A) \ar{r}{\eta_A} \ar[hook]{d}[']{p^* i^*(m)} & q^* j^* (A) \ar[hook]{d}{q^* j^* (m)} \ar{r}{\theta_A} & r^* i^* (A) \ar[hook]{d}{r^* i^* (m)} \\
p^* i^*(A) \ar{r}{\eta_X} & q^* j^* (X) \ar{r}{\theta_X} & r^* i^* (X),
\end{tikzcd}
\]
the morphism $q^* j^* (m)$ is invertible.  Also, by \thref{horizontal-factor-for-subtopos}, $\theta \eta$ admits a horizontal factorisation $\theta' \eta': p \Rightarrow r$, and so because $p$ is e.c., the outside square is also a pullback. We aim to show that $p^* i^* (m)$ is an isomorphism. This is equivalent to saying that there is a factorisation of subobjects
\[
\begin{tikzcd}
    p^* i^* (X) \ar[bend right, equals]{rd} \ar[dashed]{r} & p^* i^* (A) \ar[hook]{d}{p^* i^* (m)} \\
    & p^* i^* (X).
\end{tikzcd}
\]
Such a factorisation is obtained via the universal property of the pullback, as in the diagram
\[
\begin{tikzcd}
    & q^* j^*(X) \ar{rr}{(q^*j^*(m))^{-1}} && q^*j^*(A) \ar{d}{\theta_A} \\
    p^*i^*(X) \ar[bend right, equals]{rd} \ar[dashed]{r} \ar[bend left]{ru}{\eta_X} & p^*i^*(A) \ar[hook]{d}{p^*i^*(m)} \ar{rr}{\theta_A \eta_A} \arrow["\lrcorner"{anchor=center, pos=0.125}, draw=none]{rrd} && r^*i^*(A) \ar[hook]{d}{r^*i^*(m)} \\
    & p^*i^*(X) \ar{rr}{\theta_X \eta_X} && r^*i^*(X),
\end{tikzcd}
\]
where the necessary commutativity condition is given by
\[
r^*i^*(m) \circ \theta_A \circ (q^*j^*(m))^{-1} \circ \eta_X =
\theta_X \circ q^*j^*(m) \circ (q^*j^*(m))^{-1} \circ \eta_X =
\theta_X \circ \eta_X.
\]
Hence, $\F \xrightarrow{p} \E \xhookrightarrow{i} \G$ factors as $\F \xrightarrow{p'} \E' \xhookrightarrow{j} \G$.

We now show that $p': \F \to \E'$ is e.c. Given a natural transformation $\eta: p' \Rightarrow q$ to an arbitrary $q: \F \to \E'$, we will show that $\eta$ is an immersion. By $\F$-entwinement there exists a diagram
\[
\begin{tikzcd}[row sep=scriptsize]
	&& \E \\
	\\
	\F && {\E'} && \G\\
	\\
	&& \E
	\arrow["\cong"{description}, draw=none, from=1-3, to=3-3]
	\arrow["i", curve={height=-12pt}, hook, from=1-3, to=3-5]
	\arrow["p", curve={height=-12pt}, from=3-1, to=1-3]
	\arrow[""{name=0, anchor=center, inner sep=0}, "{p'}", curve={height=-18pt}, from=3-1, to=3-3]
	\arrow[""{name=1, anchor=center, inner sep=0}, "q"', curve={height=18pt}, from=3-1, to=3-3]
	\arrow["r"', curve={height=12pt}, from=3-1, to=5-3]
	\arrow["j", hook, from=3-3, to=3-5]
	\arrow["\theta", Rightarrow, from=3-3, to=5-3]
	\arrow["i"', curve={height=12pt}, hook, from=5-3, to=3-5]
	\arrow["\eta", shorten <=5pt, shorten >=5pt, Rightarrow, from=0, to=1]
\end{tikzcd}
\]
and hence, for any $A \leq X$ in $\G$, a commuting diagram
\[
\begin{tikzcd}
p^*i^*(A) \arrow[d, hook] &[-30pt] \cong p'^*j^*(A) \arrow[r, "\eta_{j^*(A)}"]  & q^*j^*(A) \arrow[d, hook] \arrow[r, "\theta_A"] & r^*i^*(A) \arrow[d, hook] \\
p^*i^*(X) & \cong p'^*j^*(X) \arrow[r, "\eta_{j^*(X)}"]                & q^*j^*(X) \arrow[r, "\theta_X"] & r^*i^*(X) 
\end{tikzcd}
\]
in $\F$. The outer square is a pullback since $p$ is e.c., and so the left square is a pullback because the vertical arrows are monomorphisms. We conclude by noting that because $\E'$ is a subtopos of $ \G$, and therefore a reflective subcategory, we may assume that any subobject inclusion in $\E'$ is of the form $j^*(A) \leq j^*(X)$ for $A \leq X$ in $\G$ (see also \thref{sheafification-surjective-on-subobjects}).

We have thus shown that if $p: \F \to \E$ is an e.c.\ geometric morphism then $\F \xrightarrow{p} \E \xhookrightarrow{i} \G$ factors as $\F \xrightarrow{p'} \E' \xhookrightarrow{j} \G$ with $p'$ also e.c. It follows that, as subtoposes of $\G$, the image $p(\F)$ is included in $\E'_{\ec[\F]}$ and so $\E_{\ec[\F]}$ is included in $\E'_{\ec[\F]}$. By symmetry, we have that the other inclusion also holds, hence $\E_{\ec[\F]} \simeq \E'_{\ec[\F]}$.

Conversely, suppose that $\E_{\ec[\F]} \simeq \E'_{\ec[\F]}$. By \cite[Scholium B3.4.4]{johnstone_sketches_2002} we can form the pushout of $\E \hookleftarrow \E_{\ec[\F]} \simeq \E'_{\ec[\F]} \hookrightarrow \E'$ to obtain $\E \hookrightarrow \G \hookleftarrow \E'$. Now let $p: \F \to \E$ be any geometric morphism. By \thref{every-geometric-morphism-continues-to-ec} there is an e.c.\ geometric morphism $q: \F \to \E$ together with a natural transformation $\eta: p \Rightarrow q$. By the definition of $\E_{\ec[\F]}$, we have that $q$ factors as $\F \xrightarrow{q'} \E_{\ec[\F]} \hookrightarrow \E$. We thus have that $\F \xrightarrow{q'} \E_{\ec[\F]} \hookrightarrow \E \hookrightarrow \G$ is the same as $\F \xrightarrow{q'} \E'_{\ec[\F]} \hookrightarrow \E' \hookrightarrow \G$, and so we can view $\eta$ as a natural transformation from $\F \xrightarrow{p} \E \hookrightarrow \G$ to the latter, as depicted below.
\[\begin{tikzcd}
	&[15pt] &[-30pt] &[-10pt]  \E \\
	\F & {\E_{\ec[\F]}} & {\simeq \E'_{\ec[\F]}} && \G \\
	&&& {\E'}
	\arrow[curve={height=-6pt}, hook, from=1-4, to=2-5]
	\arrow[""{name=0, anchor=center, inner sep=0}, "p", curve={height=-30pt}, from=2-1, to=1-4]
	\arrow["q'"', from=2-1, to=2-2]
	\arrow[curve={height=-6pt}, hook, from=2-3, to=1-4]
	\arrow[curve={height=6pt}, hook, from=2-3, to=3-4]
	\arrow[curve={height=6pt}, hook, from=3-4, to=2-5]
	\arrow["\eta", shorten <=4pt, Rightarrow, from=0, to=2-2]
\end{tikzcd}\]
We conclude by symmetry that $\E$ and $\E'$ are $\F$-entwined.
\end{proof}
\begin{corollary}
    For any two toposes $\E$ and $\F$ we have ${\left(\E_{\ec[\F]}\right)}_{\ec[\F]} \simeq \E_{\ec[\F]}$.
\end{corollary}
\begin{proof}
    This is a simple application of \thref{E-and-E_Fec-are-entwined} and \thref{entwinement-characterisations}.
\end{proof}

\section{Strongly existentially closed geometric morphisms}
\label{sec:sec-geometric-morphisms}
We now consider a topos-theoretic generalisation of strongly existentially closed models. Unlike e.c.\ models, the notion of being an s.e.c.\ model is highly dependent on the particular presentation of the theory (i.e.\ it is not a Morita invariant property, see \thref{ec-sec-examples}). We will therefore need to include extra information about a choice of generating set and subobject basis for our topos. We will show in \thref{sec-for-coherent} that, in the case of coherent toposes, this choice can be removed.
\begin{definition}
\thlabel{sec-geometric-morphism}
Let $\E$ be a topos, and let $\X$ be a generating set of objects for $\E$. A \emph{subobject basis} (for $\X$) is an indexed set $\B = (\B_X)_{X \in \X}$ where, for each $X \in \X$, $\B_X$ is a set of subobjects of $X$ such that $\B_X$ generates $\Sub(X)$.

For a generating set $\X$ and a subobject basis $\B$ we say that a geometric morphism $p: \F \to \E$ is \emph{strongly existentially closed with respect to $(\X, \B)$}, or \emph{$(\X, \B)$-s.e.c.}, if for each subobject $A \leq X$ in $\B$
\[
    p^*(A \lor \neg A) = p^*(X).
\]
For such subobjects $A \leq X$, we say that $A$ is \emph{complemented along} $p$.

We say $p$ is \emph{strongly existentially closed} (or \emph{s.e.c.} for short) if $p$ is $(\X, \B)$-s.e.c.\ for some generating set $\X$ and subobject basis $\B$.
\end{definition}
In \thref{sec-as-geometric-morphism-is-sec-as-model} we will see the motivation for this definition of s.e.c., and how it coincides with the classical notion from \thref{existentially-closed-model}.
\begin{remark}
\thlabel{complemented-along-p-remark}
Here, $\neg A$ denotes the pseudo-complement of $A$ in $\Sub(X)$ (i.e.\ the subobject $A \rightarrow \bot$).  For any geometric morphism $p: \F \to \E$, it is easily checked that $p^*(\neg A) \leq \neg p^*(A)$. The condition $p^*(A \lor \neg A) = p^*(X)$ is equivalent to requiring that $p^*(A)$ is complemented in $\Sub(p^*(X))$, and that $p^*(\neg A)$ is that complement, in particular $p^*(\neg A) = \neg p^*(A)$.
\end{remark}
\begin{lemma}
\thlabel{sec-stable-under-postcomposition}
Let $\E$ be some topos, let $\X$ be a generating set for $\E$ and let $\B$ be a subobject basis.
\begin{enumerate}[label=(\roman*)]
\item If $p: \F \to \E$ is $(\X, \B)$-s.e.c.\ then any composite $pq: \G \to \F \to \E$ is $(\X, \B)$-s.e.c.\ as well.
\item If $\{q_i: \G \to \F\}_{i \in I}$ is a jointly conservative family of geometric morphisms and $pq_i$ is $(\X, \B)$-s.e.c.\ for all $i \in I$ then $p$ is $(\X, \B)$-s.e.c.
\end{enumerate}
\end{lemma}
\begin{proof}
The first statement is obvious from the definition: for each $X \in \X$, the subobject lattice $\Sub(X)$ is generated by subobjects $A \leq X$ in $\B_X$ such that $p^*(A \lor \neg A) = p^*(X)$, and so $q^* p^* (A \lor \neg A) = q^* p^*(X)$ as well.

For the second statement we note that $q_i^* p^*(A \vee \neg A) = q_i^* p^*(X)$ for all subobjects $A \leq X$ in $\B_X$ and all $i \in I$. So since $\{q_i: \G \to \F\}_{i \in I}$ is jointly conservative, we have that $p^*(A \vee \neg A) = p^*(X)$ for all such $A \leq X$.
\end{proof}
\begin{definition}
\thlabel{sec-subtopos}
We call a subtopos $i: \F \hookrightarrow \E$ \emph{$(\X, \B)$-s.e.c.}\ if the embedding $i$ is $(\X, \B)$-s.e.c. We call a subtopos s.e.c.\ if it is $(\X, \B)$-s.e.c.\ for some $(\X, \B)$.
\end{definition}
\begin{corollary}
\thlabel{geometric-morphism-sec-iff-factor-through-sec}
Let $\E$ be some topos, let $\X$ be a generating set for $\E$ and let $\B$ be a subobject basis. A geometric morphism $p: \F \to \E$ is $(\X, \B)$-s.e.c.\ if and only if $p$ factors through a $(\X, \B)$-s.e.c.\ subtopos of $\E$.
\end{corollary}
\begin{proof}
If $\E' \hookrightarrow \E$ is $(\X, \B)$-s.e.c., then the composite $\F \to \E' \hookrightarrow \E$ is $(\X, \B)$-s.e.c.\ by \thref{sec-stable-under-postcomposition}(i). Conversely, consider the image factorisation $\F \twoheadrightarrow p(\F) \hookrightarrow \E$ of $p$. The composite is $(\X, \B)$-s.e.c.\ and $\F \twoheadrightarrow p(\F)$ is surjective, so by \thref{sec-stable-under-postcomposition}(ii) $p(\F) \hookrightarrow \G$ is $(\X, \B)$-s.e.c.
\end{proof}
As suggested by the terminology, being s.e.c.\ is stronger property than being e.c.  See \thref{ec-sec-examples} for examples of e.c.\ morphisms that are not s.e.c.
\begin{proposition}
\thlabel{geometric-morphism-sec-implies-ec}
Any s.e.c.\ geometric morphism $p: \F \to \E$ is also e.c.
\end{proposition}
The proof of \thref{geometric-morphism-sec-implies-ec} is an application of the following lemma.
\begin{lemma}
\thlabel{complemented-implies-square-is-pb}
Let $p: \F \to \E$ be a geometric morphism and suppose that $A \leq X$ is complemented along $p$. Then for any geometric morphism $q: \F \to \E$ and any natural transformation $\eta: p \Rightarrow q$ the naturality square below is a pullback.
\[
\begin{tikzcd}
p^*(A) \arrow[r, "\eta_A"] \arrow[d, hook] & q^*(A) \arrow[d, hook] \\
p^*(X) \arrow[r, "\eta_X"]                 & q^*(X)                
\end{tikzcd}
\]
\end{lemma}
\begin{proof}
Recall from \thref{complemented-along-p-remark} that the requirement $p^*(X) = p^*(A \lor \neg A)$ ensures that $p^*(A)$ is complemented in $\Sub(p^*(X))$ and that $p^*(\neg A)$ is that complement.

We aim to show that $p^*(A) = \eta^{-1}_X(q^*(A))$. One inequality, $p^*(A) \leq \eta^{-1}_X(q^*(A))$, is provided by the commutativity of the naturality square. We also have that
\[
    p^*(\neg A) \land \eta^{-1}_X(q^*(A)) \leq
    \eta^{-1}_X(q^*(\neg A)) \land \eta^{-1}_X(q^*(A)) =
    \eta^{-1}_X q^*(\neg A \land A) = 0.
\]
Therefore, we compute that
\begin{align*}
    \eta^{-1}_X(q^*(A)) &= (p^*(A) \lor p^*(\neg A)) \land \eta^{-1}_X(q^*(A)) \\
    &= (p^*(A) \land \eta^{-1}_X(q^*(A))) \lor (p^*(\neg A) \land \eta^{-1}_X(q^*(A))) \\
    &= p^*(A) \land \eta_X^{-1}(q^*(A)).
\end{align*}
Hence, $\eta_X^{-1}(q^*(A)) \leq p^*(A)$, concluding the proof.
\end{proof}
\begin{proof}[Proof of \thref{geometric-morphism-sec-implies-ec}]
Let $(\X, \B)$ be such that $p$ is $(\X, \B)$-s.e.c, and let $\eta: p \Rightarrow q$ be any natural transformation into some geometric morphism $q: \F \to \E$. By \thref{ec-on-generating-is-enough} it suffices to show that, for every subobject $A \leq X$ with $X \in \X$, the naturality square as in \thref{complemented-implies-square-is-pb} is a pullback.

By assumption, there is a family of subobjects $\{A_i \leq X\}_{i \in I}$ in $\B_X$ that are complemented along $p$ such that $A = \bigvee_{i \in I} A_i$. Using \thref{complemented-implies-square-is-pb}, we conclude that
\[
\eta^{-1}_X(q^*(A)) =
\eta^{-1}_X \left( q^* \left( \bigvee_{i \in I} A_i \right) \right) =
\bigvee_{i \in I} \eta^{-1}_X(q^*(A_i)) =
\bigvee_{i \in I} p^*(A_i) =
p^*(A)
\]
as desired.
\end{proof}
When the codomain topos of a geometric morphism is coherent, we can safely omit the choice of generating set and subobject basis since there exists a canonical choice.
\begin{theorem}
\thlabel{sec-for-coherent}
Let $\E$ be a coherent topos and suppose that $p: \F \to \E$ is s.e.c.  Then for any coherent $X$ and any compact subobject $A \leq X$, we have that $A$ is complemented along $p$. In particular, taking $\E_\coh$ to be the set of coherent objects in $\E$ and $\E_\comp$ to be the set of compact subobjects of coherent objects, we have that $p$ is $(\E_\coh, \E_\comp)$-s.e.c.
\end{theorem}
\begin{proof}
The ``in particular'' bit follows from the former because the coherent objects generate $\E$ and because in a coherent topos, for any object $X$, its compact subobjects generate $\Sub(X)$.

Let $(\X,\B)$ be such that $p$ is $(\X,\B)$-s.e.c. We may assume that $\X$ is closed under finite coproducts and subobjects. This can be done by closing $\X$ first under finite coproducts and then under subobjects, as coproducts of subobjects are subobjects of coproducts. We quickly argue why these operations preserve the property that $p$ is $(\X,\B)$-s.e.c.

Let $\{X_i\}_{i \in I}$ be a (finite)\footnote{In fact, we may take any cardinality, showing that we could close $\X$ under coproducts of size $< \kappa$ for any cardinal $\kappa$, but we are only interested in finite coproducts here.} set of objects in $\X$. For each $i \in I$, $\B_{X_i} $ is a family $ \{A_j \leq X_i\}_{j \in J_i}$ of subobjects that generates $\Sub(X_i)$, each of which is complemented along $p$. Composing with the coprojection $X_i \hookrightarrow \coprod_{i \in I} X_i$, we may also view this as a family of subobjects of $\coprod_{i \in I} X_i$. Using the fact that coproducts in a Grothendieck topos are \emph{disjoint}, one straightforwardly checks that, for each $j \in J_i$, $A_j$ is also complemented along $p$ as a subobject of $\coprod_{i \in I} X_i$. Letting $i$ vary, we thus conclude that $\{A_j \leq \coprod_{i \in I} X_i \mid i \in I, j \in J_i\}$ is a family of subobjects that generates $\Sub(\coprod_{i \in I} X_i)$, each of which is complemented along $p$.  For closure under subobjects, let $Y \leq X$ with $X \in \X$. %Let $\{A_i \leq X\}_{i \in I}$ be generating $\Sub(X)$, with each $A_i$ complemented along $p$. 
Then $\{A \wedge Y \leq Y \mid A \leq X \text{ in } \B_X \}$ is generating for $\Sub(Y)$ and each member is complemented along $p$.

Let $A \leq X$ be a compact subobject of a coherent object, as in the statement. Then $X$ admits a covering $\{B_i \to X\}_{i \in I}$ by objects in $\X$, and each $B_i$ admits a covering $\{C_j \to B_i\}_{j \in J_i}$ by coherent objects. For each $i \in I$ and each $j \in J_i$, the image $C_j \twoheadrightarrow B'_j \hookrightarrow B_i$ is compact, since $B'_j$ has a cover $C_j \twoheadrightarrow B'_j$ by a compact object (see \cite[Lemma D3.3.3]{johnstone_sketches_2002}). Thus, $A$ admits a covering by compact objects in $\X$ (as we assumed $\X$ to be closed under subobjects). Moreover, as $X$ is coherent, this covering can be taken to finite, and hence by taking the coproduct $B = \coprod B'_j$, there is an epimorphism $k: B \twoheadrightarrow X$ whose domain is a compact object in $\X$ (by \cite[Lemma D3.3.3]{johnstone_sketches_2002} and the assumption that $\X$ is closed under finite coproducts).

Since $A$ and $B$ are both compact, we have that the pullback
\[
\begin{tikzcd}
    k^{-1}(A) \ar[two heads]{r} \ar[hook]{d} & A \ar[hook]{d} \\
    B \ar[two heads]{r}{k} & X
\end{tikzcd}
\]
is compact too. By assumption, there is a family $\{A_i\}_{i \in I}$ of subobjects of $B$ such that $k^{-1}(A) = \bigvee_{i \in I} A_i$ and $A_i$ is complemented along $p$ for all $i \in I$. Since $k^{-1}(A)$ is compact, we may assume $I$ to be finite. Using general facts about Heyting algebras, we have that
\[
\neg k^{-1}(A) =
\neg \bigvee_{i \in I} A_i =
\bigwedge_{i \in I} \neg A_i.
\]
We thus compute that
\[
p^*(k^{-1}(A) \vee \neg k^{-1}(A)) =
p^* \left( \bigvee_{i \in I} A_i \vee \bigwedge_{i \in I} \neg A_i \right) =
\bigwedge_{i \in I} \bigvee_{j \in I} p^*(A_j  \vee \neg A_i) =
p^*(B),
\]
and so we see that $k^{-1}(A)$ is complemented along $p$.

The maps $k^{-1}: \Sub(X) \to \Sub(B)$ and $p^*(k)^{-1}: \Sub(p^*(X)) \to \Sub(p^*(B))$ are open frame homomorphisms (see \cite[Exercise III.16]{maclane_sheaves_1994}). In particular, they are morphisms of Heyting algebras. Using this, together with the fact that $p$ preserves pullbacks and that $k^{-1}(A)$ is complemented along $p$, we compute that
\begin{align*}
p^*(B) &= p^*(k^{-1}(A) \vee \neg k^{-1}(A)), \\
&= p^*(k)^{-1}(p^*(A)) \vee p^*(k)^{-1}(p^*(\neg A)), \\
&= p^*(k)^{-1}(p^*(A \vee \neg A)).
\end{align*}
Since $k$ is an epimorphism, $p^*(k)$ is also an epimorphism, and so $p^*(k)^{-1}$ is injective. Since $p^*(B) = p^*(k^{-1}(X)) = p^*(k)^{-1}(p^*(X))$, we thus conclude that $p^*(A \vee \neg A) = p^*(X)$, and so $A$ is complemented along $p$, as required.
\end{proof}
\begin{corollary}
\thlabel{sec-as-geometric-morphism-is-sec-as-model}
A model of a coherent theory $T$ in $\Set$ is s.e.c.\ according to \thref{existentially-closed-model} if and only if the corresponding geometric morphism $p: \Set \to \Set[T]$ is s.e.c.\ according to \thref{sec-geometric-morphism}.
\end{corollary}
\begin{proof}
For the right to left direction we could string together \thref{geometric-morphism-sec-implies-ec,ec-as-model-is-ec-as-point,classical-ec-is-sec}. However, a direct proof will add intuition to \thref{sec-geometric-morphism}.

Let $\X$ be the set of interpretations of coherent formulas in the generic model $G_T$ in $\Set[T]$. For each $\phi(G_T) \in \X$, we take $\B_{\phi(G_T)}$ as the set of subobjects $\psi(G_T) \leq \phi(G_T)$ where $\psi$ is also a coherent formula, which are precisely the subobjects of $\phi(G_T)$ that are compact. By \thref{sec-for-coherent}, it suffices to show that $M$ is s.e.c.\ (in the sense of \thref{existentially-closed-model}) if and only if every $\psi(G_T) \leq \phi(G_T)$, with $\phi$ and $\psi$ coherent, is complemented along $p$.

We can compute the pseudo-complement of $\psi(G_T)$ in $\Sub(\phi(G_T))$, as
\[
\bigvee \{ \chi(G_T) \mid \chi(x) \text{ a geometric formula and } \psi(G_T) \land \chi(G_T)  = 0 \text{ in } \Sub(\phi(G_T)) \},
\]
which can be further reduced to
\[
\bigvee \{ \chi(G_T) \mid \chi(x) \text{ a coherent formula and } T \text{ proves }  \psi(x) \land \chi(x) \vdash_x \bot \}.
\]
So $\psi(G_T) \leq \phi(G_T)$ is complemented along $p$ if and only if the sequent
\[
\phi(x) \vdash_x \psi(x) \vee \bigvee \{ \chi(x) \text{ a coherent formula} \mid T \text{ proves } \psi(x) \land \chi(x) \vdash_x \bot \}
\]
holds in $M$.  We conclude by noting the the above sequent holding for every pair of coherent formulas $\phi(x)$ and $\psi(x)$ exactly axiomatises $M$ being s.e.c.\ (in the sense of \thref{existentially-closed-model}, see also \thref{geometric-axiomatisation-always-possible}).
\end{proof}
\begin{example}
\thlabel{ec-sec-examples}
We consider various (non-)examples of being s.e.c.\ or e.c.\ for models in $\Set$.  Once the notion of a locally zero-dimensional topos has been introduced, we will give an example of a geometric morphism that is e.c.\ but not s.e.c\ with respect to any choice of generating set or subobject basis in \thref{ec-but-never-sec}.
\begin{enumerate}[label=(\roman*)]
\item Let $T$ be any of the coherent theories from \thref{classical-ec-examples}. In that example we classified the e.c.\ models of $T$. So, by \thref{ec-as-model-is-ec-as-point,sec-as-geometric-morphism-is-sec-as-model}, a geometric morphism $p: \Set \to \Set[T]$ is e.c.\ and s.e.c.\ if and only if the corresponding model is e.c.\ (and hence s.e.c., by \thref{classical-ec-is-sec}). We will see additional examples where being e.c.\ and s.e.c.\ coincide in toposes different from $\Set$ in \thref{automorphism-examples}.
\item Let $X$ be a topological space. Then we can view $X$ as a propositional theory $T_X$ by introducing a propositional symbol for each open set of $X$, and have the geometric propositional theory $T_X$ capture unions and intersections, so that $\Sh(X) \simeq \Set[T_X]$. An element $x \in X$ becomes a $\Set$-model of $T_X$ by setting $x \models U$ if and only if $x \in U$. A homomorphism $x \to y$ between two elements $x, y \in X$ is just an inequality $x \leq y$ in the \emph{specialisation order} on $X$ (i.e., for all opens $U \subseteq X$, if $x \in U$ then $y \in U$). Let us now assume that $X$ is sober, so its $\Set$-models correspond precisely to its elements. Then a point $x \in X$ is e.c.\ if, whenever $x \leq y$, then $y \in U$ also implies $x \in U$, and so $y \leq x$. In other words, a point $x \in X$ is e.c.\ if and only if the point $x$ is \emph{maximal} in the pre-order $(X, \leq)$. In particular, in a sober T$_1$ space all points are e.c.
\item We continue the previous point. Suppose that $X$ is a zero-dimensional space. Let $\X$ be the set of subterminal objects in $\Sh(X)$, and let $\B$ be the set of complemented subobjects of subterminal objects. As $X$ is zero-dimensional, we have that $\B$ is a subobject basis for the generating set $\X$. As a result, every point on $\Sh(X)$ is $(\X, \B)$-s.e.c.

If $X$ is locally zero-dimensional, but not necessarily zero-dimensional (e.g., the space $Y$ in \thref{sheaves-on-zero-dimensional-locale-are-not-zero-dimensional}) we can adjust the above by taking $\X$ to consist of all those subterminal objects in $\Sh(X)$ that correspond to open zero-dimensional subspaces of $X$. Again, we take $\B$ to be the set of complemented subobjects of objects in $\X$, and then every point on $\Sh(X)$ is $(\X, \B)$-s.e.c.

When $X$ is not locally zero-dimensional, being $(\X, \B)$-s.e.c.\ heavily depends on the choice of $\X$ and $\B$. At least when $X$ is Hausdorff, such a choice can always be made. Fix some $x \in X$. Let $\U_0 = \{U \subseteq X \text{ open} \mid x \in U\}$ and $\U_1 = \{U \subseteq X \text{ open} \mid x \in \neg U\}$, where $\neg U$ denotes the interior of $X \setminus U$, i.e.\ the pseudo-complement of $U$ in the Heyting algebra of opens of $X$. By Hausdorffness, $\U_0 \cup \U_1$ is a basis for $X$. Once again, we take $\X$ to be the set of subterminal objects in $\Sh(X)$ and we let $\B$ be the set of subobjects of subterminals corresponding to an open in $\U_0 \cup \U_1$. The geometric morphism $p: \Set \to \Sh(X)$ corresponding to $x$ is then $(\X, \B)$-s.e.c. Indeed, let $V \in \X$. If $x \not \in V$ then $p^*$ evaluates $V$ and all its subobjects to $0$. Otherwise, let $U \leq V$ in $\B_V$, we either have $U \in \U_0$, and so $p^*(U \vee \neg U) = p^*(U) \vee p^*(\neg U) = 1 \vee 0 = 1 = p^*(V)$, or we have $U \in \U_1$, and so $p^*(U \vee \neg U) = p^*(U) \vee p^*(\neg U) = 0 \vee 1 = 1 = p^*(V)$.

The above choice heavily depends on the point $x \in X$. For example, when $X = \R$ with the usual Euclidean topology, then the above choice will be different for different points. One easily checks that if $(\X, \B)$ is as above for some $x \in \R$, then for any $x' \neq x$ we have that the morphism $p': \Set \to \Sh(X)$ corresponding to $x'$ is not $(\X, \B)$-s.e.c.
\end{enumerate}
We note that the above shows that there can be no version of \thref{every-geometric-morphism-continues-to-ec} for $(\X, \B)$-s.e.c. That is, any of the above examples of an e.c.\ model that is not $(\X, \B)$-s.e.c.\ cannot admit a homomorphism into an $(\X, \B)$-s.e.c.\ model, as that homomorphism would then be an immersion (since the domain is e.c.) and one quickly verifies that this implies that the domain then has to be $(\X, \B)$-s.e.c.\ as well.
\end{example}

\section{Locally zero-dimensional toposes}
\label{sec:locally-zero-dimensional-toposes}
The axiomatisation of s.e.c.\ models described in \thref{geometric-axiomatisation-always-possible} is reminiscent of the notion of zero-dimensionality in topology.  Indeed, \thref{ec-subtopos-characterisation} will solidify this connection.  To that end, in this section we introduce the notion of a locally zero-dimensional topos.

\subsection{Zero-dimensional locales}
Before defining locally zero-dimensional toposes, we recall some facts regarding zero-dimensional and locally zero-dimensional locales. Zero-dimensional locales extend the classical notion of a zero-dimensional space from topology, i.e.\ those spaces with a basis of clopen subsets.
\begin{definition}
\thlabel{zd-locale}
A locale $L$ is called \emph{zero-dimensional} if every element is a join of complemented elements. A locale $L$ is called \emph{locally zero-dimensional} if there exists an open cover $\U \subseteq L$ such that, for every $U \in \U$, the open sub-locale $U \hookrightarrow L$, that is the locale $U = \{V \in L \mid V \leq U\}$, is zero-dimensional.
\end{definition}
As one would expect, a topological space is (locally) zero-dimensional if and only if its locale of open sets is (locally) zero-dimensional.

Recall that a morphism of locales $f: X \to Y$ is a \emph{local homeomorphism} if $X$ admits an open covering $\U$ such that, for each $U \in \U$ there is $V \in Y$ such that $f$ restricts to a homeomorphism $U \to V$ (see e.g.\ \cite[Exercise IX.9]{maclane_sheaves_1994}).
\begin{lemma}
\thlabel{epic-local-homeomorphism-preserves-locally-zero-dimensional}
Suppose that $f: X \to Y$ is a local homeomorphism of locales. If $Y$ is locally zero-dimensional, then so is $X$. Conversely, if $f$ is epic and $X$ is locally zero-dimensional, then $Y$ is locally zero-dimensional too.
\end{lemma}
\begin{proof}
Suppose that $Y$ is locally zero-dimensional. Then $X$ admits an open covering $\U$ such that each $U \in \U$ is homeomorphic to an open sub-locale of $Y$, which is therefore locally zero-dimensional.

Conversely, suppose that $f$ is epic and $X$ is locally zero-dimensional.  There is an open covering $\U$ of $X$ such that $f(U)$ is open and $f|_U: U \to f(U)$ is a homeomorphism for each $U \in \U$.  Since zero-dimensionality is inherited by sub-locales, each $U \in \U$ admits a covering by opens that are zero-dimensional, so we may as well assume that each $U \in \U$ is zero-dimensional, and thus $f(U)$ is zero-dimensional too.  Finally, as $f$ is epic, $\{ f(U) \mid U \in \U \}$ describes an open covering of $Y$.
\end{proof}
\begin{example}
\thlabel{sheaves-on-zero-dimensional-locale-are-not-zero-dimensional}
The statement in \thref{epic-local-homeomorphism-preserves-locally-zero-dimensional} cannot be modified by replacing `locally zero-dimensional' with `zero-dimensional'. The idea of this counterexample is due to Peter Johnstone, and appears in \cite[Example 3.3]{bunge_quasi_2007}. Let $X = \N \cup \{\infty\}$ be the one-point compactification of the natural numbers (which carry the discrete topology) and let $Y = (X + X) / {\sim}$ be the quotient space that is obtained by identifying the two copies of $\N$. The evident quotient maps
\[
\begin{tikzcd}
    X + X \ar[two heads]{r} & Y \ar[two heads]{r} & X
\end{tikzcd}
\]
are both epic local homeomorphisms.  Moreover, $X$ and $X+X$ are zero-dimensional, but $Y$ is only locally zero-dimensional.
\end{example}
\subsection{Locally zero-dimensional toposes}
\begin{definition}
We call an object $X$ in a topos $\E$ a \emph{zero-dimensional object} (resp., \emph{locally zero-dimensional object}) if $\Sub(X)$ is a zero-dimensional (resp., locally zero-dimensional) as a locale.
\end{definition}
Our definition of locally zero-dimensional topos requires a choice of a generating set of objects, but like we have seen for s.e.c.\ geometric morphisms, there is a canonical choice if the topos is coherent (see \thref{locally-zero-dimensional-coherent-topos}).
\begin{definition}
\thlabel{locally-zero-dimensional-topos}
We say that a topos $\E$ is \emph{locally zero-dimensional} if there is a generating set of zero-dimensional objects.
\end{definition}
\begin{remark}
Evidently, a topos is locally zero-dimensional if and only if there is a generating set of locally zero-dimensional objects, because each locally zero-dimensional object is covered by zero-dimensional objects.
\end{remark}
The following justifies our terminology of locally zero-dimensional topos.
\begin{proposition}
\thlabel{sheaves-on-locally-zero-dimensional-locale-are-locally-zero-dimensional}
For a locale $L$, the following are equivalent:
\begin{enumerate}[label=(\roman*)]
\item\label{L_is_0dim} $L$ is a locally zero-dimensional locale,
\item\label{ShL_is_0dim} $\Sh(L)$ is a locally zero-dimensional topos,
\item\label{all_subobj_of_ShL_is_0dim} every object in $\Sh(L)$ is locally zero-dimensional.
\end{enumerate}
\end{proposition}
\begin{proof}
Clearly, using $\Sub(1) \cong L$ and that the subterminals in $\Sh(L)$ generate the topos, \ref{L_is_0dim} implies \ref{ShL_is_0dim} and \ref{all_subobj_of_ShL_is_0dim} implies \ref{L_is_0dim}.  It remains to show that \ref{ShL_is_0dim} implies \ref{all_subobj_of_ShL_is_0dim}.  We will use the description of $\Sh(L)$ as the slice ${\bf LH}/L$ of local homeomorphisms over $L$. For each local homeomorphism $f: Y \to L$, there is a diagram
\[
\begin{tikzcd}
    Y & \ar[two heads]{l} U \ar[hook]{r} & \coprod_{i \in I} L,
\end{tikzcd}
\]
in which both arrows are local homeomorphisms (see, for instance, \cite[Theorem C1.4.7]{johnstone_sketches_2002}).  The coproduct $\coprod_{i \in I} L$ is locally zero-dimensional since $L$ is locally zero-dimensional; hence, the sub-locale $U$ inherits this local zero-dimensionality. Finally, $Y$ is locally zero-dimensional by an application of \thref{epic-local-homeomorphism-preserves-locally-zero-dimensional}.
\end{proof}
\begin{remark}
Recall from \thref{sheaves-on-zero-dimensional-locale-are-not-zero-dimensional} that there exists a local homeomorphism $Y \to L$ where $L$ is zero-dimensional but $Y$ is only {locally zero-dimensional}. Thus, a locale $L$ being zero-dimensional does \emph{not} imply that every subobject frame in $\Sh(L)$ is zero-dimensional.
\end{remark}
\begin{question}
\thlabel{locally-zero-dimensional-question}
For any topos $\E$, it quickly follows that if all objects in $\E$ are locally zero-dimensional then $\E$ is locally zero-dimensional. Does the converse hold in general?
\end{question}
\begin{theorem}
\thlabel{locally-zero-dimensional-coherent-topos}
Let $\E$ be a locally zero-dimensional coherent topos, then every coherent object in $\E$ is zero-dimensional.
\end{theorem}
\begin{proof}
We could prove the statement directly.  But we have already proved a stronger version in the form of \thref{sec-for-coherent}.  We describe how to massage a locally zero-dimensional coherent topos into the form needed by \thref{sec-for-coherent}.

Let $\X$ be the generating set of zero-dimensional objects, and let $\B$ be the set of subobjects of objects in $\X$ that are complemented. Then the identity functor $\id_\E$ is $(\X, \B)$-s.e.c., because being complemented along $\id_\E$ is equivalent to just being complemented. We can thus apply \thref{sec-for-coherent} to see that $\id_\E$ is $(\E_\coh, \E_\comp)$-s.e.c. This precisely means that compact subobjects of coherent objects are complemented. As such subobjects generate $\Sub(X)$ in a coherent topos, we conclude that $X$ is indeed zero-dimensional.
\end{proof}
The relevance of local zero-dimensionality for us is the following.
\begin{proposition}
\thlabel{geometric-morphism-into-locally-zero-dimensional-is-sec}
Any geometric morphism $p: \F \to \E$ into a locally zero-dimensional topos $\E$ is s.e.c. More precisely, if $\X$ is a generating set of zero-dimensional objects in $\E$ then, letting $\B_X$ be the set of complemented subobjects of $X \in \X$, any geometric morphism $p: \F \to \E$ is $(\X, \B)$-s.e.c.
\end{proposition}
\begin{proof}
This amounts to writing out definitions: any subobject $A \leq X$ in $\B_X$ is complemented in $\Sub(X)$, so $A \vee \neg A = X$. Therefore, such a subobject is complemented along any geometric morphism into $\E$.
\end{proof}
\begin{example}\thlabel{ec-but-never-sec}
    We return to the problem of finding an e.c.\ geometric morphism that is not s.e.c\ with respect to any generating set and subobject basis.  Let $X$ be a sober, $T_1$ space that is not locally zero-dimensional and consider the identity geometric morphism $\id_{\Sh(X)}$.  Since $X$ is $T_1$, $\id_X$ is maximal in the specialisation order on endomorphisms on $X$, and so for the same reason as in \thref{ec-sec-examples}(ii), $\id_{\Sh(X)}$ is an e.c.\ geometric morphism.  But to say that $\id_{\Sh(X)}$ is s.e.c.\ (with respect to some generating set and subobject basis) is to say that $\Sh(X)$ is locally zero-dimensional, which is not the case by \thref{sheaves-on-locally-zero-dimensional-locale-are-locally-zero-dimensional}.
\end{example}

\section{The classifying topos of e.c.\ models}
\label{sec:classifying-topos-ec-models}
In this section, we return to the study of the subtopos $\E_{\ec[\F]} \hookrightarrow \E$ of e.c.\ $\F$-points.  We give a characterisation of $\E_{\ec[\F]}$ under the assumption that the e.c.\ morphisms $\F \to \E$ are all $(\X,\B)$-s.e.c.\ for some choice of $(\X,\B)$. Thus, we will recover our motivating case, where $\E$ classifies a coherent theory and $\F = \Set$, but we will also see other examples where $\F$ does not have to be taken as $\Set$.
\begin{theorem}
\thlabel{ec-subtopos-characterisation}
Let $\F$ and $\E$ be any two toposes. Suppose that for some generating set $\X$ of $\E$, with subobject basis $\B$, we have that every e.c.\ geometric morphism $p:  \F \to \E$ is $(\X, \B)$-s.e.c. Then the following are equivalent for a subtopos $\G \hookrightarrow \E$ with enough $\F$-points:
\begin{enumerate}[label=(\roman*)]
\item $\G \simeq \E_{\ec[\F]}$,
\item $\G$ is locally zero-dimensional and $\F$-entwined with $\E$,
\item $\G$ is the maximal $(\X, \B)$-s.e.c.\ subtopos of $\E$ with enough $\F$-points.
\end{enumerate}
\end{theorem}
\begin{remark}
We note that, by a simple application of \thref{geometric-morphism-sec-iff-factor-through-sec}, the assumptions of \thref{ec-subtopos-characterisation} are necessary as well as sufficient.  Specifically, if $\E_{\ec[\F]}$ is an $(\X,\B)$-s.e.c.\ subtopos of $\E$, then every e.c.\ geometric morphism $p: \F \to \E$ is $(\X,\B)$-s.e.c.
\end{remark}
In order to prove the theorem, we must show that, under the assumptions of the theorem, $\E_{\ec[\F]}$ is $(\X,\B)$-s.e.c./locally zero-dimensional.  This is a consequence of the following results.
\begin{lemma}
\thlabel{sheafification-surjective-on-subobjects}
Let $i: \F \hookrightarrow \E$ be a geometric embedding. Then the induced map on subobjects $i^*: \Sub(X) \to \Sub(i^*(X))$ is surjective for all $X$ in $\E$.
\end{lemma}
\begin{proof}
Let $B \leq i^*(X)$ be a subobject.  We first form the following pullback in $\E$:
\[
\begin{tikzcd}
A \arrow[r] \arrow[d, hook] & i_*(B) \arrow[d, hook] \\
X \arrow[r, "\eta_X"]      & i_* i^*(X)               
\end{tikzcd}
\]
where $\eta$ is the unit of the adjunction $i^* \dashv i_*$. Then we take the transpose of that square under this adjunction, which is pictured on the left below, and decomposes as the diagram on the right below:
\[
\begin{tikzcd}
    i^*(A) \ar[hook]{d} \ar{r} & B \ar[hook]{d} \\
    i^*(X) \ar[equal]{r} & i^*(X),
\end{tikzcd}
\quad
\begin{tikzcd}
    i^*(A) \ar{r} \ar[hook]{d} & i^*i_*(B) \ar{r}{\varepsilon_B} \ar[hook]{d} & B \ar[hook]{d} \\
    i^*(X) \ar{r}{i^*(\eta_X)} & i^*i_*i^*(X) \ar{r}{\varepsilon_{i^*(X)}} & i^*(X),
\end{tikzcd}
\]
where $\varepsilon$ is the counit of the adjunction $i^* \dashv i_*$. In the diagram on the right both squares are pullbacks: the left square because it is the image of a pullback under the finite limit preserving functor $i^*$, and the right square because $\varepsilon$ is an isomorphism. The transposed square is thus a pullback, and so the arrow $i^*(A) \to B$ is an isomorphism. We conclude that $i^*(A) = B$ as subobjects of $i^*(X)$.
\end{proof}
\begin{proposition}
\thlabel{sec-subtopos-is-locally-zero-dimensional}
If $i: \F \hookrightarrow \E$ is an s.e.c.\ subtopos, then $\F$ is locally zero-dimensional.
\end{proposition}
\begin{proof}
Let $(\X, \B)$ be such that the inclusion $i: \F \hookrightarrow \E$ is $(\X, \B)$-s.e.c. Then $i^*(\X) = \{i^*(X) \mid X \in \X\}$ is a generating set for $\F$, so it suffices to show that $\Sub(i^*(X))$ is zero-dimensional for each $X \in \X$. For each $A \leq X$ in $\B_X$ we have that $i^*(A)$ is complemented in $\Sub(i^*(X))$. We thus conclude by noting that $\{i^*(A) \mid A \leq X \text{ in } \B_X\}$ generates $\Sub(i^*(X))$, as $i^*$ induces a surjection $\Sub(X) \to \Sub(i^*(X))$ by \thref{sheafification-surjective-on-subobjects} and it preserves joins of subobjects.
\end{proof}
\begin{proposition}
\thlabel{sec-is-ec-then-ec-subtopos-is-sec}
Let $\F$ and $\E$ be any two toposes. Suppose that for some generating set $\X$ of $\E$, with subobject basis $\B$, we have that every e.c.\ geometric morphism $p: \F \to \E$ is $(\X, \B)$-s.e.c. Then $\E_{\ec[\F]} \hookrightarrow \E$ is $(\X, \B)$-s.e.c., and so in particular $\E_{\ec[\F]}$ is locally zero-dimensional.
\end{proposition}
\begin{proof}
By construction, $\E_{\ec[\F]}$ has enough $\F$-points, and each $\F \to \E_{\ec[\F]} \hookrightarrow \E$ is $(\X, \B)$-s.e.c.\ by assumption. So $\E_{\ec[\F]} \hookrightarrow \E$ is $(\X, \B)$-s.e.c.\ by \thref{sec-stable-under-postcomposition}(ii), and the final remark then follows from \thref{sec-subtopos-is-locally-zero-dimensional}.
\end{proof}

\begin{proof}[Proof of \thref{ec-subtopos-characterisation}]
We first prove (i) $\Leftrightarrow$ (ii). The left to right direction follows from \thref{sec-is-ec-then-ec-subtopos-is-sec,entwinement-characterisations}. For the converse we note that, as a subtopos of itself, $\G = \bigcup \{ p(\F) \mid p: \F \to \G \text{ a geometric morphism}\}$ because $\G$ has enough $\F$-points. As $\G$ is locally zero-dimensional, by \thref{geometric-morphism-into-locally-zero-dimensional-is-sec,geometric-morphism-sec-implies-ec}, every geometric morphism $\F \to \G$ is e.c., and so we have that $\G \simeq \G_{\ec[\F]}$. Finally, since $\G$ and $\E$ are $\F$-entwined, by \thref{entwinement-characterisations} we conclude that $\G \simeq \G_{\ec[\F]} \simeq \E_{\ec[\F]}$.

We prove (i) $\Leftrightarrow$ (iii) by showing that $\E_{\ec[\F]}$ is the maximal $(\X, \B)$-s.e.c.\ subtopos of $\E$ with enough $\F$-points. Firstly, $\E_{\ec[\F]}$ has enough $\F$-points by construction and is $(\X, \B)$-s.e.c.\ by \thref{sec-is-ec-then-ec-subtopos-is-sec}. Let $\G \hookrightarrow \E$ be any $(\X, \B)$-s.e.c.\ subtopos with enough $\F$-points. From being s.e.c.\ it follows from \thref{sec-stable-under-postcomposition,geometric-morphism-sec-implies-ec} that any composite $\F \to \G \hookrightarrow \E$ is e.c. So by construction of $\E_{\ec[\F]}$, any such composite factors through $\E_{\ec[\F]}$. As before, from enough $\F$-points it follows that $\G = \bigcup \{ p(\F) \mid p: \F \to \G \text{ a geometric morphism}\}$, and so we conclude that $\G \hookrightarrow \E$ factors through $\E_{\ec[\F]}$.
\end{proof}
\begin{corollary}
\thlabel{ec-subtopos-characterisation-for-entwined}
Suppose that a topos $\E'$ is $\F$-entwined with a topos $\E$ that satisfies the assumptions of \thref{ec-subtopos-characterisation}. Then for a subtopos $\G \hookrightarrow \E'$ we have $\G \simeq \E'_{\ec[\F]}$ if and only if $\G$ has enough $\F$-points, is locally zero-dimensional and is $\F$-entwined with $\E'$.
\end{corollary}
\begin{proof}
By \thref{entwinement-characterisations} we have that $\E'_{\ec[\F]} \simeq \E_{\ec[\F]}$, and so we simply apply \thref{ec-subtopos-characterisation}.
\end{proof}
The hypotheses of \thref{ec-subtopos-characterisation} are satisfied in our original motivating case of e.c.\ models in $\Set$ of some coherent theory $T$. We discuss the simplifications of \thref{ec-subtopos-characterisation} that can be made.
\begin{remark}
\thlabel{ec-subtopos-characterisation-coherent}
If $\E$ is a coherent topos, then by \thref{sec-for-coherent} we do not need to worry about the generating set $\X$ and subobject basis $\B$ as these can be taken to be $\E_\coh$ and $\E_\comp$ respectively. So the assumptions of \thref{ec-subtopos-characterisation} simplify to ``every e.c.\ geometric morphism $\F \to \E$ is s.e.c.''. In the conclusion of that theorem, we can now also simplify (iii) to characterise $\E_{\ec[\F]}$ as the maximal s.e.c.-subtopos of $\E$ with enough $\F$-points.
\end{remark}
After \thref{ec-subtopos} we discussed how $\Set[T]_{\ec[\F]}$ is the classifying topos of the common geometric theory $T^{\ec[\F]}$ of the e.c.\ models of $T$ in $\F$. We simplify the notation for the case where $\F = \Set$ and $T$ is coherent.
\begin{definition}
\thlabel{t-ec}
Let $T$ be a coherent theory. We define $T^{\ec}$ to be the common geometric theory of all e.c.\ models of $T$ in $\Set$. That is:
\[
T^{\ec} = \{ \sigma \text{ a geometric sequent} \mid \sigma \text{ is valid in every e.c.\ model of $T$ in $\Set$}\}.
\]
In particular, $T^{\ec}$ axiomatises the class of e.c.\ models of $T$ in $\Set$ (see \thref{geometric-axiomatisation-always-possible}).
\end{definition}
\begin{repeated-theorem}[\thref{ec-subtopos-characterisation-coherent-set}]
Let $T$ be a coherent theory. Then following are equivalent for a subtopos $\G \hookrightarrow \Set[T]$ with enough points:
\begin{enumerate}[label=(\roman*)]
\item $\G \simeq \Set[T^{\ec}]$,
\item $\G$ is locally zero-dimensional and $\Set$-entwined with $\Set[T]$,
\item $\G$ is the maximal s.e.c.-subtopos of $\Set[T]$ with enough points.
\end{enumerate}
\end{repeated-theorem}
\begin{proof}
Note that $\Set[T^{\ec}] \simeq \Set[T]_{\ec[\Set]}$ by the discussion after \thref{ec-subtopos}. Then the statement is just a simplification of \thref{ec-subtopos-characterisation}, following \thref{ec-subtopos-characterisation-coherent} and using \thref{classical-ec-is-sec}, which states that e.c.\ models of $T$ in $\Set$ are s.e.c.
\end{proof}
\subsection{Models with an endomorphism}
Of course, the generality of \thref{ec-subtopos-characterisation} allows for more applications than just the case $\F = \Set$. We finish this section by giving a class of such examples when $\F$ is taken to be the topos $\Set^\N$ or $\Set^\Z$ (where $\N$ and $\Z$ are considered as monoids). In particular, a model of a theory $T$ in $\Set^\N$ (resp.\ $\Set^\Z$) is a pair $(M, \sigma)$ consisting of a $\Set$-model $M$ of $T$ and an endomorphism (resp.\ automorphism) $\sigma: M \to M$ of the model.
\begin{proposition}
\thlabel{automorphism-ec-iff-underlying-ec}
Let $T$ be a coherent theory with the amalgamation property for $\Set$-models, i.e.\ any span of homomorphisms $M_1 \leftarrow M_0 \to M_2$ of $\Set$-models can be completed to a commuting square. Then, for a model $(M, \sigma)$ of $T$ in $\Set^\N$, the following are equivalent:
\begin{enumerate}[label=(\roman*)]
\item $(M, \sigma)$ is e.c.\ in $\Set^\N$,
\item $(M, \sigma)$ is s.e.c.\ in $\Set^\N$,
\item $M$ is e.c.\ is $\Set$,
\item $M$ is s.e.c.\ in $\Set$.
\end{enumerate}
The same is true when $\Set^\N$ is replaced by $\Set^\Z$.
\end{proposition}
\begin{proof}
The implications (ii) $\Rightarrow$ (i) and (iii) $\Leftrightarrow$ (iv) follow from \thref{geometric-morphism-sec-implies-ec,classical-ec-is-sec}. Furthermore, (iv) $\Rightarrow$ (ii) easily follows from forgetting $\sigma$, as the interpretation of geometric formulas remains the same (equivalently stated, we apply \thref{sec-stable-under-postcomposition}(ii) with the surjective geometric morphism $\Set \twoheadrightarrow \Set^\N$).

We prove (i) $\Rightarrow$ (iv) by showing that $M$ immerses into a s.e.c.\ $\Set$-model of $T$, which straightforwardly implies that $M$ is s.e.c.\ too.  We note that there exists a homomorphism of models $f:  M \to N$ where $N$ is an $|M|^+$-homogeneous and e.c.\ model of $T$ (and so $N$ is also s.e.c.\ by \thref{classical-ec-is-sec}). Recall that an e.c.\ model $N$ is $\kappa$-homogeneous if for any two tuples $a$ and $b$ of the same length $< \kappa$, we have that $\tp(a; N) = \tp(b; N)$ implies that there is an automorphism $\tau: N \to N$ that sends $a$ to $b$ (pointwise). Here, $\tp(a; N)$ is the model-theoretic notion of (positive) type: the set of all coherent formulas satisfied by $a$ in $N$. Such a homomorphism $f: M \to N$ can be constructed by the usual model-theoretic methods (see e.g.\ \cite[Section 6.1]{tent_course_2012}).  We must show that $f$ is an immersion.

By our assumption on $T$, the span $N \xleftarrow{\, f \, } M \xrightarrow{f \sigma} N$ can be amalgamated to obtain a commuting square of $\Set$-models
\[
\begin{tikzcd}
    N \ar{r}{g} & N^* \\
    M \ar{u}{f} \ar{r}{f \sigma} & N. \ar{u}[']{h}
\end{tikzcd}
\]
Let $m$ be an (infinite) tuple that enumerates the model $M$. As $N$ is e.c., the homomorphisms $g$ and $h$ are immersions, and so they both preserve and reflect the satisfaction of coherent formulas. Therefore, we have that
\[
\tp(f(m); N) =
\tp(gf(m); N^*) =
\tp(hf\sigma(m); N^*) =
\tp(f\sigma(m); N),
\]
and so by homogeneity there is an automorphism $\tau: N \to N$ sending $f(m)$ to $f \sigma(m)$. As $m$ enumerated all of $M$, this is equivalent to saying that $\tau f = f \sigma$.  Now viewing $(N, \tau)$ as a model of $T$ in $\Set^\N$, it follows that $f: (M, \sigma) \to (N, \tau)$ is a homomorphism of models internal to $\Set^\N$. Since $(M,\sigma)$ was assumed to be e.c.\ in $\Set^\N$, this means that $f$ is an immersion as desired, thus concluding the proof that $M$ is s.e.c.

The same proof applies if $\sigma$ is assumed to be an automorphism, and since the $\tau$ we construct is an automorphism, the statement for $\Set^\Z$ also follows.
\end{proof}
\begin{corollary}
    If $T$ is a coherent theory with the amalgamation property for $\Set$-models, the hypotheses of \thref{ec-subtopos-characterisation} are satisfied for $\E = \Set[T]$ and $\F = \Set^\N$ or $\Set^\Z$.
\end{corollary}
\begin{example}
\thlabel{automorphism-examples}
We give some examples of coherent theories with the amalgamation property for $\Set$-models.
\begin{enumerate}[label=(\roman*)]
\item The theory of fields has the amalgamation property. It follows that a difference field (i.e., a field with an endomorphism) is e.c.\ in $\Set^\N$ (or $\Set^\Z$) if and only if the underlying field is algebraically closed (see \thref{classical-ec-examples}(ii)). Difference fields have been studied model-theoretically before by adding the endomorphism to the language \cite{macintyre-generic-automorphism-1997,chatzidakis-difference-fields-1999}. Note that the e.c.\ models of a field with an endomorphism in the language differ from the e.c.\ models of fields in $\Set^\N$.  This is because, in the former case, since the endomorphism forms part of the language, an e.c.\ model also needs to contain solutions to certain difference polynomials (i.e., polynomials where the endomorphism can be applied to the variables, called \emph{basic systems} in \cite{macintyre-generic-automorphism-1997}).
\item The theory of algebraically closed exponential fields. These are algebraically closed fields with a group homomorphism from the additive group to the multiplicative group, called the exponential map. This is clearly coherently axiomatisable, and any span of such fields can be amalgamated by \cite[Theorem 4.3]{haykazyan_existentially_2021}. The requirement that the underlying field is algebraically closed is important for the amalgamation property. If we would just consider exponential fields then the amalgamation property fails in a strong sense: for any exponential field $F_0$ whose underlying field is not algebraically closed, we can find a span $F_1 \leftarrow F_0 \to F_2$ of exponential fields that cannot be amalgamated (this is also \cite[Theorem 4.3]{haykazyan_existentially_2021}).
\item The theory of decidable graphs, see also \thref{classical-ec-examples}(iii), has the amalgamation property, given by the pushout of graphs. If we remove the requirement that the graphs are decidable we still have the amalgamation property, but the theory trivialises from the point of view of e.c.\ models: the only e.c.\ model is the graph with one vertex (i.e., the terminal object in the category of graphs).
\end{enumerate}
\end{example}

\section{Atomic e.c.\ models and double negation sheaves}
\label{sec:atomic-ec-models-and-double-negation-sheaves}
As an application of the above, we replicate \cite[Theorem 6]{blass_classifying_1983}, using a different proof, to characterise those points on a coherent topos that factorise through the double negation sheaves on that topos (\thref{points-factorising-through-double-negation}). It turns out, these are precisely the e.c.\ models that are \emph{atomic} (see \thref{atomic}).
\subsection{Denseness and e.c.\ models}
The topos of double negation sheaves $\Sh_{\neg \neg}(\E)$ is the smallest dense subtopos of $\E$ (see \cite[Corollary A4.5.20]{johnstone_sketches_2002}). Therefore, we first relate the notion of a dense subtopos to our ongoing study of the subtopos $\E_{\ec[\F]} \hookrightarrow \E$ of e.c.\ $\F$-points.
\begin{definition}
\thlabel{dense-subtopos}
Recall that a subtopos $i: \F \hookrightarrow \E$ is called \emph{dense} if $i^*(X) \cong 0$ implies $X \cong 0$ for all $X \in \E$.
\end{definition}
\begin{lemma}
\thlabel{dense-pseudo-negation}
If $i: \F \hookrightarrow \E$ is dense, then $i^*$ preserves pseudo-complements. That is, for any subobject $A \leq X$ in $\E$ we have $i^*(\neg A) = \neg i^*(A)$ in $\Sub(i^*(X))$.
\end{lemma}
\begin{proof}
We can explicitly compute $\neg A$ as $\bigvee \{ B \leq X \mid B \wedge A \cong 0 \}$. Similarly, using that $i^*: \Sub(X) \to \Sub(i^*(X))$ is surjective (\thref{sheafification-surjective-on-subobjects}), we have that $\neg i^*(A) = \bigvee \{i^*(B) \mid B \leq X \text{ and } i^*(B) \wedge i^*(A) \cong 0 \}$. We deduce the result because $i^*(B) \wedge i^*(A) \cong i^*(B \wedge A) \cong 0$ if and only if $B \wedge A \cong 0$, by denseness.
\end{proof}
\begin{proposition}
\thlabel{dense-simplifies-sec}
Let $\E$ be any topos and let $\X$ be a generating set for $\E$ with subobject basis $\B$. Then a dense subtopos of $\E$ is $(\X, \B)$-s.e.c.\ if and only if $i^*(A)$ is complemented in $\Sub(i^*(X))$ for every subobject $A \leq X$ in $\B_X$.
\end{proposition}
\begin{proof}
Recall from \thref{complemented-along-p-remark} that a subobject $A \leq X$ in $\E$ being complemented along $i$ is equivalent to $i^*(A)$ being complemented in $\Sub(i^*(X))$ and $i^*(\neg A) = \neg i^*(A)$. The statement thus follows from \thref{dense-pseudo-negation}.
\end{proof}
\begin{proposition}
\thlabel{ec-subtopos-dense}
If a topos $\E$ has enough $\F$-points, then $i: \E_{\F\text{-}\ec} \hookrightarrow \E$ is dense.
\end{proposition}
\begin{proof}
Let $X$ be any object in $\E$ and suppose that $i^*(X) \cong 0$.  We must show that $X \cong 0$. Since $\E$ has enough $\F$-points, it suffices to show that $p^*(X) \cong 0$ for each geometric morphism $p: \F \to \E$. By \thref{every-geometric-morphism-continues-to-ec}, there is a natural transformation $\eta: p \Rightarrow q$ where $q$ is an e.c.\ geometric morphism, which therefore factors through $i$, and so $q^*(X) \cong 0$. Then there is an arrow $\eta_X: p^*(X) \to q^*(X) \cong 0$. As the initial object in a topos is strict (see \cite[Lemma A1.4.1]{johnstone_sketches_2002}), this implies that $p^*(X) \cong 0$.
\end{proof}
\begin{corollary}
\thlabel{ec-subtopos-dense-coherent-set}
For a coherent theory $T$, the inclusion $\Set[T^{\ec}] \hookrightarrow \Set[T]$ is dense.
\end{corollary}
\begin{proof}
\thref{ec-subtopos-dense} applies, because $T$ is coherent, so $\Set[T]$ has enough points.
\end{proof}
\begin{remark}
\thlabel{t-ec-same-universal-consequences}
By viewing an arbitrary topos $\E$ as the classifying topos $\Set[T]$ of some geometric theory $T$, we can view a subtopos $\F \hookrightarrow \E$ as the classifying topos $\Set[T']$ of some \emph{quotient theory} $T'$ of $T$ \cite[Theorem 3.2.5]{caramello_theories_2017}. In particular, following \cite[Section 4.2.4]{caramello_theories_2017}, a subtopos $\Set[T'] \hookrightarrow \Set[T]$ is dense if and only if, for every regular formula $\phi(x)$, we have that the sequent $\phi(x) \vdash_x \bot$ (i.e.\ $\forall x \neg \phi(x)$) is derivable from $T$ if and only if it is derivable from $T'$.

Let now $T$ be a coherent theory. By combining the above with \thref{ec-subtopos-dense-coherent-set} we arrive at the well-known fact that the common universal theory of e.c.\ models of $T$ in $\Set$ is the same as the universal consequences of $T$. That is, for any regular $\phi(x)$, we have that $\forall x \neg \phi(x)$ holds in every e.c.\ model of $T$ in $\Set$ if and only if $\forall x \neg \phi(x)$ can be derived from $T$.
\end{remark}
\subsection{Atomic e.c.\ models}
We now turn to replicating \cite[Theorem 6]{blass_classifying_1983}.
\begin{definition}
\thlabel{atomic}
We call a model $M$ in $\Set$ of a (coherent) theory $T$ \emph{atomic} if each finite tuple in $M$ is \emph{supported}, i.e.\ for each $a \in M$, there is a geometric formula $\phi(x)$ with $M \models \phi(a)$ such that, for each geometric formula $\psi(x)$ with $M \models \psi(a)$, we have that $\phi(x) \vdash_x \psi(x)$ is derivable from $T$. We say that $\phi(x)$ \emph{supports} $a$.
\end{definition}
We could have equivalently restricted ourselves to regular formulas in \thref{atomic} for both $\phi(x)$ and $\psi(x)$, because every geometric formula is equivalent to a disjunction of regular formulas.
\begin{remark}
    We also briefly note that what we call ``supported'' is called ``isolated'' in classical model theory, because it means exactly that the \emph{type} of $a$ is an isolated point in the space of types. However, in positive model theory this correspondonce no longer holds, and so ``isolated'' would not be a good term. See also the discussion after \cite[Definition 5.1]{kamsma_bilinear_2023}.
\end{remark}
\begin{remark}
In \cite[Pages 135--136]{blass_classifying_1983} the notion in \thref{atomic} is called ``strongly atomic''. As is explained there, this is to distinguish from the notion where we replace ``$\phi(x) \vdash_x \psi(x)$ is derivable from $T$'' with ``$\phi(x) \vdash_x \psi(x)$ is valid in $M$''. We chose this terminology because it coincides with modern terminology \cite{haykazyan_spaces_2019, kamsma_bilinear_2023,caramello_atomic_2012}. Even though \cite[Definition 5.4]{kamsma_bilinear_2023} focuses only on coherent logic, \thref{atomic} is still equivalent, because we may restrict to regular formulas, as mentioned above. Our notion of atomic is the same as that from \cite[Definition 6.1]{haykazyan_spaces_2019}, as is explained in \cite[Definition 5.2]{kamsma_bilinear_2023}. If we add the requirement that the model is e.c., then our notion coincides with \cite[Definition 3.15]{caramello_atomic_2012}, at least for coherent theories, because the requirement there that every type realised in $M$ is ``complete'' is equivalent to saying that $M$ is e.c.
\end{remark}
\begin{example}
\thlabel{atomic-not-ec}
Being atomic alone does not imply being e.c. For example, let $T$ be the empty theory in the empty language. Then its models in $\Set$ are pure sets, and the singletons are e.c.\ models. However, every non-empty set is an atomic model. For a tuple $a$ in some non-empty set $M$ this is witnessed by the formula $\phi$, which is a conjunction of all the equalities between the singletons in $a$, together with $\exists x(x=x)$.
\end{example}
\begin{theorem}
\thlabel{points-factorising-through-double-negation}
Let $T$ be a coherent theory. Then a point $p: \Set \to \Set[T]$ factors through the double negation subtopos $\Sh_{\neg \neg}(\Set[T])$ if and only if the model $M$ in $\Set$ corresponding to $p$ is an atomic e.c.\ model (atomic with respect to $T^{\ec}$).
\end{theorem}
One step in the proof immediately follows from \thref{dense-simplifies-sec} above:
\begin{lemma}
\thlabel{factor-through-double-negation-is-ec}
    If $p: \Set \to \Set[T]$ factors through $\Sh_{\neg \neg}(\Set[T])$, then the corresponding model $M$ is e.c.
\end{lemma}
\begin{proof}
    Since $\Sh_{\neg \neg}(\Set[T]) \hookrightarrow \Set[T]$ is a dense, Boolean subtopos, by \thref{dense-simplifies-sec}, it is an s.e.c.\ subtopos.  The result now follows by \thref{geometric-morphism-sec-iff-factor-through-sec,geometric-morphism-sec-implies-ec}.
\end{proof}
The remainder of the proof is based on a useful fact (\thref{sub-open-factorisation}) about sub-open geometric morphisms, for which we use the characterisation in \cite[Lemma 1.7]{johnstone_open_1980} as a definition.
\begin{definition}
\thlabel{sub-open}
Recall from \cite[Proposition V.1.1]{joyal-tierney-1984} that a frame homomorphism $h: X \to Y$ is called \emph{open} if the following equivalent conditions hold:
\begin{enumerate}[label=(\roman*)]
    \item $h$ preserves the Heyting implication,
    \item $h$ is a homomorphism of complete Heyting algebras,
    \item $h$ has a left adjoint satisfying the \emph{Frobenius identity}.
\end{enumerate}
A geometric morphism $f: \F \to \E$ is then called \emph{sub-open} if for every object $X$ in $\X$ the induced frame homomorphism $f^*: \Sub(X) \to \Sub(f^*(X))$ is open.
\end{definition}
\begin{fact}[{\cite[Proposition 3.6]{johnstone_open_1980}}]
\thlabel{sub-open-factorisation}
A geometric morphism $f: \F \to \E$, with $\F$ a Boolean topos, is sub-open if and only if it factors through $\Sh_{\neg \neg}(\E)$.
\end{fact}
\begin{lemma}\thlabel{sub-open-on-generators-enough}
    It suffices to check sub-openness on a generating set. Explicitly: let $\E$ be a topos with a choice of generating set $\X$ and let $f: \F \to \E$ be a geometric morphism.  If, for each $X \in \X$, the induced map $f^*: \Sub(X) \to \Sub(f^*(X))$ is open, then $f^*: \Sub(Y) \to \Sub(f^*(Y))$ is open for all objects $Y$ in $\E$.
\end{lemma}
\begin{proof}
    The exists a jointly epimorphic family of morphisms $\{g_i: X_i \to Y\}_{i \in I}$ in $\E$ with $X_i \in \X$. Thus, there exists a commuting square of frames
    \[
    \begin{tikzcd}
        \Sub(Y) \ar{r}{f^*} \ar{d}[']{g_i^{-1}} & \Sub(f^*(Y)) \ar{d}{f^*(g_i)^{-1}} \\
    \Sub(X_i) \ar{r}{f^*} & \Sub(f^*(X_i)),
    \end{tikzcd}
    \]
    where the vertical maps are open and yield jointly injective families $\{g_i^{-1}\}_{i \in I}$ and $\{f^*(g_i)^{-1}\}_{i \in I}$. Therefore, we compute that, for each $i \in I$ and $A,B \leq Y$,
    \begin{align*}
        f^*(g_i)^{-1} f^*(A \rightarrow B) & = f^* g_i^{-1}(A \rightarrow B) , \\
        & = f^* g_i^{-1}(A) \rightarrow  f^* g_i^{-1}(B) , \\
        & = f^*(g_i)^{-1} f^*(A) \rightarrow f^*(g_i)^{-1} f^*(B), \\
        & = f^*(g_i)^{-1}( f^*(A) \rightarrow f^*(B)).
    \end{align*}
    Now using that the family of maps $\{f^*(g_i)^{-1}\}_{i \in I}$ is jointly injective, we conclude that $f^*(A \rightarrow B) =  f^*(A) \rightarrow f^*(B)$ as desired.
\end{proof}
\begin{proof}[Proof of \thref{points-factorising-through-double-negation}]
    First consider the case where $p: \Set \to \Set[T]$ factors through $\Sh_{\neg \neg}(\Set[T])$.  By \thref{factor-through-double-negation-is-ec}, it remains to show that $M$ is atomic. As $p$ is sub-open, $p^*$ is an open map on subobjects, in particular $p^*$ preserves arbitary meets of subobjects. Therefore, for any $a \in M$ we have
    \begin{align*}
        \bigcap \{\psi(M) \mid M \models \psi(a)\} & = \bigcap \{p^*(\psi(G_T)) \mid M \models \psi(a) \} \\
        & = p^*\left(\bigwedge \{ \psi(G_T) \mid M \models \psi(a) \}\right).
    \end{align*}
    But $\bigwedge \{ \psi(G_T) \mid M \models \psi(a) \}$ is a subobject of the generic model $G_T$ in $\Set[T]$, and every such subobject is of the form $\phi(G_T)$. It is now easily checked that $\phi(x)$ supports $a$.

    Conversely, suppose that $M$ is both e.c.\ and atomic, which entails that $M$ is s.e.c.\ and atomic by \thref{classical-ec-is-sec}.  We aim to show that $p^*$ preserves Heyting implication of subobjects.  By \thref{sub-open-on-generators-enough}, it suffices to show that $p^*$ preserves Heyting implication of subobjects of the form $\phi(G_T)$, since the interpretations of formulas in the generic model generate the topos $\Set[T]$.  That is, we aim to show that
    \[
        \phi(M) \rightarrow \psi(M) = \{a \in M \mid M \not \models \phi(a) \text{ or } M \models \psi(a)\}
    \]
    is equal to
    \[
        p^*(\phi(G_T) \rightarrow \psi(G_T)) = \bigcup \{ \chi(M) \mid T \text{ proves } \phi(x) \land \chi(x) \vdash_x \psi(x) \}.
    \]
    One inclusion, namely that $p^*(\phi(G_T) \rightarrow \psi(G_T))$ is included in $\phi(M) \rightarrow \psi(M)$, follows from the universal property of the Heyting implication. We prove the other inclusion.
    
    Clearly, if $M \models \psi(a)$, then tautologically $T$ proves $\phi(x) \land \psi(x) \vdash_x \psi(x)$ and so $a \in \psi(M) \subseteq \phi(M) \rightarrow \psi(M)$. Suppose instead that $M \not \models \phi(a)$.  Since $M$ is atomic, there is some formula $\xi(x)$ that supports $a$. We claim that $T$ proves $\phi(x) \land \xi(x) \vdash_x \psi(x)$, and hence $a \in \phi(M) \rightarrow \psi(M)$.  First, express $\phi(x)$ as a disjunction $\bigvee_{i \in I} \phi_i(x)$ of coherent formulas. Now let $i \in I$ be arbitrary. As $M$ is s.e.c.\ and $M \not \models \phi_i(a)$, there is some $\zeta_i(x)$ such that $M \models \zeta_i(a)$ and $T$ proves $\phi_i(x) \land \zeta_i(x) \vdash_x \bot$. By the choice of $\xi(x)$, $T$ proves $\xi(x) \vdash_x \zeta_i(x)$, and so $T$ also proves $\phi_i(x) \wedge \xi(x) \vdash_x \bot$. As $i \in I$ was arbitrary, the chain of derivations
    \[
    \phi(x) \land \xi(x) \vdash_x {\textstyle \bigvee_{i \in I} \phi_i(x) \land \xi(x) } \vdash_x \bot \vdash_x \psi(x)
    \]
    can be derived from $T$.  Thus, $p$ is sub-open and so factors through $\Sh_{\neg \neg}(\Set[T])$ as desired.
\end{proof}

\section{Countably categorical theories and atomic toposes}
\label{sec:countably-categorical-theories-and-atomic-toposes}
Finally,  using \cite{caramello_atomic_2012}, we relate atomic e.c.\ models to the notion of countable categoricity from positive model theory \cite{haykazyan_spaces_2019, kamsma_bilinear_2023}. For this section, we revert to our convention that a ``model'' intends a model in $\Set$.
\begin{definition}
\thlabel{countably-categorical}
A coherent theory $T$ is called \emph{countably categorical} if, up to isomorphism, there is one (at most) countable e.c.\ model of $T$.
\end{definition}
\begin{remark}
\thlabel{different-notions-of-countable-categoricity}
We have taken the definition of countably categorical from positive model theory here (see e.g., \cite[Theorem 6.5]{haykazyan_spaces_2019} or \cite[Definition 5.1]{kamsma_bilinear_2023}). Note that this definition restricts to e.c.\ models, and so it allows $T$ to have non-isomorphic countable models. This is in contrast to for example \cite[Definition 3.14]{caramello_atomic_2012}, where a geometric theory $T$ is called countably categorical if any two countable models are isomorphic. The two major differences between these definitions are then that \cite{caramello_atomic_2012} considers \emph{all} models and that the version in \cite{caramello_atomic_2012} can be vacuously satisfied. Hence, a coherent theory $T$ is countable categorical (in our sense) if and only if $T^{\ec}$ has a countable model and is countably categorical in the sense of \cite{caramello_atomic_2012}. The requirement that $T^{\ec}$ has a countable model is automatic in most practical cases: if $T$ is a countable theory, then this follows from the downward L\"owenheim-Skolem theorem.
\end{remark}
\begin{example}
The reason for restricting ourselves to e.c.\ models in \thref{countably-categorical} is because positive model theory only studies the e.c.\ models. For example, take the theory from \thref{classical-ec-examples}(i), which consisted of a countable infinity of distinct constant symbols and nothing more. Clearly this theory has many non-isomorphic countable models: namely one for every $n \leq \omega$, which corresponds to having $n$ elements that are not the interpretation of a constant symbol. When $n = 0$ this describes the only e.c.\ model. In particular, there is only one countable e.c.\ model, and so this theory is countably categorical.

A more natural example comes from \cite[Definition 2.1]{kamsma_bilinear_2023} by considering the theory of $K$-bilinear spaces, for a fixed field $K$. If $K$ is countable then this theory turns out to be countably categorical \cite[Corollary 5.9]{kamsma_bilinear_2023}. However, there are still many non-isomorphic countable models, but these will in general not even be $K$-bilinear spaces. For example, by compactness there will be countable models with a realisation for the set of formulas $\{[x,y] \neq \lambda \mid \lambda \in K\}$.
\end{example}
\begin{remark}
    Our definition of countably categorical allows for having a unique finite model. Model-theoretically, this case is often ignored, because a complete first-order theory with a finite model has that model as its only model. Analogously, a coherent theory with JCP (see \thref{jcp}) with a finite e.c.\ model has that model as its only e.c.\ model. However, there is no harm in allowing this trivial case.
\end{remark}
\begin{definition}
\thlabel{jcp}
A coherent theory $T$ is said to have the \emph{joint continuation property}, or \emph{JCP}, if for any two of its models $M$ and $M'$, there is a third model $N$ admitting homomorphisms $M \to N \leftarrow M'$.
\end{definition}
In light of the fact that every model of $T$ admits a homomorphism into an e.c.\ model (\thref{every-model-continues-to-ec-model}), we may assume that the model $N$ above is an e.c.\ model.
\begin{lemma}
\thlabel{jcp-iff-complete}
A consistent coherent theory has JCP if and only if $T^{\ec}$ is complete in the sense of \cite[Definition 3.4]{caramello_atomic_2012}, i.e.\ every geometric sentence is $T^{\ec}$-provably equivalent to either $\top$ or $\bot$, but not both.
\end{lemma}
\begin{proof}
Suppose that $T$ has JCP. Let $\phi$ be any geometric sentence and let $M$ be an e.c.\ model of $T$. Let $M'$ be any other e.c.\ model of $T$. Then there is a model $N$ admitting homomorphisms $M \to N \leftarrow M'$. As $M$ and $M'$ are both e.c., these homomorphisms are immersions, and so we have that
\[M \models \phi \iff N \models \phi \iff M' \models \phi.\]
We thus see that the truth value of $\phi$ in $M$ determines its truth value in all e.c.\ models of $T$, and so we must have that either $\top \vdash \phi$ or $\phi \vdash \bot$ is in $T^{\ec}$.

For the converse, let $M$ and $M'$ be models of $T$. We extend our language with constant symbols for the elements of $M$ and $M'$ (which we may assume to be disjoint) and define a theory
\[
T_M = \{ \top \vdash_x \chi(a) \mid \chi(x) \text{ is atomic, } a \in M \text{ and } M \models \chi(a) \}.
\]
We define $T_{M'}$ analogously. Now consider the theory $T' = T \cup T_M \cup T_{M'}$. We will show that $T'$ has a model $N$. Such a model $N \models T'$ demonstrates that $T$ has JCP because $N \models T$, and we can define a function $M \to N$ by sending an element of $M$ to the interpretation of the corresponding constant in $N$, which is a homomorphism because $N \models T_M$. Similarly, $N \models T_{M'}$ yields a homomorphism $M' \to N$.

Note that as $T'$ is a coherent theory, we can apply a compactness argument to show that it has a model. First observe that, by taking a finite conjunction of the sentences involved, every finite sub-theory $S \subseteq T_M$ is equivalent to the theory $\{\top \vdash_x \phi(a)\}$ for a finite tuple $a \in M$ and $\phi(x)$ a conjunction of atomic formulas for which $M \models \phi(a)$. Using \thref{every-model-continues-to-ec-model}, there exists a homomorphism $f: M \to N$ with $N$ an e.c.\ model of $T$.  Thus, $N \models \phi(f(a))$ and so $N \models \exists x \phi(x)$. As $T^{\ec}$ is complete, $\exists x \phi(x)$ must be derivable from $T^{\ec}$. Now let $g: M' \to N'$ be another homomorphism with $N'$ an e.c.\ model of $T$.  Then $N' \models \exists x \phi(x)$, so let $b \in N'$ be such that $N' \models \phi(b)$. We can now expand $N'$ to a model of $T \cup T_{M'} \cup \{\top \vdash_x \phi(a)\}$ over our expanded language: for each element $c \in M'$, the corresponding constant symbol is interpreted by $g(c)$, and we interpret constant symbols corresponding to the tuple $a$ as $b$.  The interpretation of the remaining constant symbols coming from $M $ does not matter, and so we obtain a model of $T \cup T_{M'} \cup \{\top \vdash_x \phi(a)\}$. We have thus shown that $T'$ is finitely satisfiable, and so it has a model by compactness, which concludes our proof.
\end{proof}
\begin{theorem}
\thlabel{countably-categorical-iff-atomic}
Let $T$ be a coherent theory in a countable language with JCP. Then $T$ is countably categorical if and only if $\Set[T^{\ec}]$ is atomic and non-degenerate.
\end{theorem}
\begin{proof}
Firstly, by \thref{jcp-iff-complete}, $T^{\ec}$ is complete in the sense of \cite{caramello_atomic_2012}. Next, we note that, under either assumption in the result, $T^{\ec}$ has a model.  If $T$ is countably categorical, then by definition $T^{\ec}$ has a model, namely the unique countable e.c.\ model of $T$. On the other hand, $\Set[T^{\ec}]$ has enough points by construction and so it is non-degenerate if and only if it has a point, i.e.\ $T^{\ec}$ has a model.

We can thus apply \cite[Theorem 3.16]{caramello_atomic_2012}, which (once specialised to our situation) states that the following are equivalent:
\begin{enumerate}[label=(\roman*)]
\item any two countable models of $T^{\ec}$ are isomorphic and $\Set[T^{\ec}]$ is Boolean,
\item $\Set[T^{\ec}]$ is atomic,
\item every model of $T^{\ec}$ is atomic.
\end{enumerate}
We deduce the result by using the characterisation of countably categorical coherent theories \cite[Theorem 6.5]{haykazyan_spaces_2019}, which asserts that (iii) is satisfied if and only if $T$ is countably categorical.
\end{proof}
\begin{corollary}
\thlabel{infinitary-first-order-becomes-geometric-in-countably-categorical}
Let $T$ be a countably categorical coherent theory with JCP in a countable language. Then every infinitary full first-order formula in $\L_{\infty, \omega}$ is $T^{\ec}$-provably equivalent to a geometric formula, modulo the deduction-system for classical infinitary first-order logic.
\end{corollary}
\begin{proof}
By \thref{countably-categorical-iff-atomic}, $\Set[T^{\ec}]$ is atomic, and hence Boolean. Therefore, the deducibility of infinitary first-order sequents modulo the deduction-system for classical infinitary first-order logic coincides with validity in the generic model in $\Set[T^{\ec}]$ (see \cite[Theorem 5.14]{kamsma_classifying_2023}). We conclude by noting that every subobject of the generic model in $\Set[T^{\ec}]$ is the interpretation of some geometric formula.
\end{proof}

% References
\bibliographystyle{alpha}
\bibliography{bibfile}

% Index of all the notes
%\printindex[notes]

\end{document}